\documentclass[11pt,reqno]{article}
\usepackage{amsmath, latexsym, amsfonts, amssymb,
amsthm, amscd,epsfig,enumerate, float} 
\advance\hoffset -.75cm

\usepackage{amsfonts}
\usepackage{latexsym}
\usepackage{amsmath}
\usepackage{amssymb}
\usepackage{color}

\newcommand{\R}{\mathbb R}

\newcommand{\N}{\mathbb N}
\newcommand{\Prob}{\mathbb P}

\newcommand{\Zd}{{{\mathbb Z}^d}}

\newcommand{\eps}{\varepsilon}

\newcommand{\dd}{{\mathrm d}}

\newcommand{\pr}{\mathbb P}

\newcommand{\ident}{{\mathchoice {\rm 1\mskip-4mu l} {\rm 1\mskip-4mu l}
{\rm 1\mskip-4.5mu l} {\rm 1\mskip-5mu l}}}

\newtheorem{teo}{Theorem}[section]
\newtheorem{lem}[teo]{Lemma}
\newtheorem{cor}[teo]{Corollary}
\newtheorem{rem}[teo]{Remark}
\newtheorem{pro}[teo]{Proposition}
\newtheorem{defn}[teo]{Definition}

\title
{The timing of life history events in presence of soft disturbances}

\author
{Daniela Bertacchi\thanks{The first two authors contributed equally to this work.} \\
Dipartimento di Matematica e Applicazioni\\
Universit\`a di Milano--Bicocca\\
via Cozzi 53, 20125 Milano, Italy\\
daniela.bertacchi\@@unimib.it
\and
Fabio Zucca$^*$ \\
Dipartimento di Matematica,\\
Politecnico di Milano,\\
Piazza Leonardo da Vinci 32, 20133 Milano, Italy.\\
fabio.zucca\@@polimi.it
\and
Roberto Ambrosini\\
Dipartimento di Biotecnologie e Bioscienze,\\
Universit\`a di Milano--Bicocca\\
Piazza della Scienza 2-4, 20126 Milano, Italy\\
roberto.ambrosini\@@unimib.it
}

\date{}

\begin{document}

\maketitle

\begin{abstract}
We study a model for the evolutionarily stable strategy (ESS) used by biological populations for choosing the time
of life-history events, such as migration and breeding. In our model
we accounted for both intra-species competition (early individuals have a competitive advantage) and a disturbance
which strikes at a random time, killing a fraction $1-p$ of the population.
Disturbances include spells of bad weather, such as freezing or heavily raining days.
It has been shown in \cite{cf:IwLe95}, that when $p=0$, then
the ESS is a mixed strategy, where individuals wait for a certain time and afterwards
start arriving (or breeding) every day. We remove the constraint $p=0$ and show that if $0<p<1$
then the ESS still implies a mixed choice of times, but strong competition may lead to a massive arrival
at the earliest time possible of a fraction of the population, while the rest 
will arrive throughout the whole period during which the disturbance may occur. 
More precisely, given $p$, there is a threshold for the competition
parameter $a$, above which massive arrivals occur and below which there is a behaviour as in
\cite{cf:IwLe95}. We study the behaviour of the ESS and of the average fitness of the population, depending on the parameters
involved. We also discuss how the population may be affected by climate change, in two respects:
first, how the ESS should change under the new climate and whether this change implies
an increase of the average fitness; second, which is the impact of the new climate on a population
that still follows the old strategy. We show that, at least under some conditions, extreme weather
events imply a temporary decrease of the average fitness (thus an increasing mortality). 
If the population adapts to the new climate, the survivors 
may have a larger fitness.
\end{abstract}



\baselineskip .6 cm

\section{Introduction}
\label{sec.intro}

Proper timing of life-history events, like emergence, germination, migration or breeding, is crucial for survival 
and successful reproduction of almost all organisms. Timing may be set by endogenous rhythms or by extrinsic 
environmental clues (e.g.~day length or temperature, \cite{cf:SainoAmbr08}), 
but in almost all cases timing seems to have evolved according to 
two contrasting selective pressures. On the one hand, the first individuals that emerge or arrive at a given site often perform better, 
because they can profit from the better habitats and benefit from reduced competitions (at least for some time). 
On the other hand, however, early individuals may suffer from higher mortality, as they expose themselves to the 
risk of adverse environmental 
conditions, which usually are more likely early than late in the season. Autumn migration may seem an exception to this pattern, 
as the risk of mortality is probably larger for late than early departing individuals. However, in this case, migrants 
may benefit from a longer stay 
in their breeding grounds (allowing to rise a further brood or acquire larger fat reserves for migration). 
At the end, this pattern can be seen as the exact reverse of the process going on in spring, and can therefore be modelled in 
the same way.

In a seminal work Iwasa and Levin \cite{cf:IwLe95}
have provided a first theoretical description of how the risk of incurring in adverse environmental 
conditions may shape the timing of life history events of a population,
as a result of evolution over many generations. Under the assumption
that adverse conditions (\textit{disturbance} according to their 
definition, which we will follow hereafter) strike at a random time and  are so strong that no individual incurring 
in the disturbance can survive, the authors show that (in many cases) the evolutionarily
stable strategy (ESS from here on) is an asynchronous choice of times in the population.
This asynchronicity has been observed in various settings and modelled by several authors (see 
\cite{cf:BessIw12}, \cite{cf:CalFag04}, \cite{cf:LevinIwetal2001}, 
 \cite{cf:synchrrepr99}, \cite{cf:timingofreprod01}
just to mention a few).
In this paper we extend Iwasa and Levin's model to
the broader scenario where the disturbance is \textit{soft}, meaning that it kills 
each individual with probability $1-p$ (in many cases some individuals 
in the population survive even to dramatic adverse conditions). 
The model of Iwasa and Levin 
can be seen as a particular case of ours when $p=0$. 

Before going into the details of our model, we have to mention that it focuses on the 
long time behaviour of a very large population. Indeed it is implicitly assumed that
when the existence of mutants with better fitness is theoretically
possible, then such mutants will appear and spread across the population. This does not take into account the disappearance
(in finite populations) of certain alleles by mere random factors, a phenomenon which can be studied 
by means of mathematical population
genetics (see for instance \cite{cf:ewens}). Long time behaviour of finite populations can also be studied through spatial models,
namely \textit{interacting particles systems}. Space not only adds complexity \cite{cf:DurrLevin},
 but may also be interpreted as ``type'', that is the location of one or more individuals can be seen 
as representing their genotype.
For the simplest among this models, the \textit{branching random walk}, much has been done: 
for instance in
\cite{cf:BZ, cf:BZ2, cf:BZ4, cf:Muller08-2, cf:Stacey03, cf:Z1} one finds characterization of the 
persistence/disappearance of genotypes (seen as locations for the model), on general space structures;
the same can be found, for some random graphs, in \cite{cf:BZ5,  cf:PemStac1}.
Stochastic modelling and interacting particle systems have been 
successfully applied to biology and ecology 
(see \cite{cf:BBZ, cf:BLZ, cf:BPZ, cf:CFMel08, cf:CDLan09, cf:KLan12, cf:Lan13,
cf:LanNeu06, cf:Mel90, cf:GaMaSch, cf:Z2014} just to mention a few).
Although stochastic modelling is very interesting and complex, 
here we will assume that over many generations, our populations
have been sufficiently large to justify the use of a model where stochasticity appears only
in the random time at which the disturbance strikes.
%

As for terminology, in the present work the life-history events under study are arrival times, meaning that
we focus on migratory birds and the time of arrival to their breeding grounds.
This choice of words should not be considered a reduction in the scope of this paper, since 
our modelling approach is very broad, as it applies to the investigation of the timing of 
any life-history event when the benefits from being early and the risk of incurring in a soft 
disturbance are in conflict.
Migratory birds are a well studied biological system where timing is crucial 
for the fitness of individuals, and where a long record of adverse conditions killing or impairing the 
reproduction of individuals exists (see e.g.~\cite{cf:Newton2007}). 
Moreover the interest in the timing of these recurrent biological events is very strong among biologists 
after that several studies have consistently observed an advancement in arrival and reproduction of 
birds supposedly as a consequence of climate change 
(see \cite{cf:Both-et-al.2004}, \cite{cf:Cotton2003}, \cite{cf:Walther2002}). 
Climate change not only implies warming temperatures, but also higher frequency 
and intensity of extreme meteorological events (see \cite{IPCC2012}). This increased weather unpredictability 
may severely affect migrant birds, because warmer springs prompt birds towards earlier arrivals, 
while more frequent unseasonable weather increases the risk of mass mortality events (see \cite{cf:Schi05}
for a stochastic model for random catastrophes striking a spatially structured population).
It is widely accepted that climate change is endangering migrant populations (see \cite{cf:SainoAmbrProcRoy}).
Ornithologists therefore strongly need models investigating the contrasting forces affecting 
the timing of bird migration to improve their understating of the ongoing ecological processes and to plan better 
conservation strategies for declining migrant populations.
We point out that even if our primary interest here are evolutionarily stable strategies which arise in large
populations after many generations of stable climate (an equilibrium situation), our study also allows us
to analyze some effects of a sudden climate change (an off-equilibrium dynamics). Indeed our results
not only describe the ESS (Theorem~\ref{th:implicit}) but also the effects of the climate change on the fitness
of a population not yet adapted to the change 
(Propositions~\ref{pro:timeintervals}, \ref{pro:averagefitnessclimatechange}
and \ref{pro:changepaveragefitness}).

Here is a short outline of the
results of the paper.
In Section~\ref{sec:model} we introduce the model and the notation.
We consider an intra-species competition regulated by a parameter $a \ge 0$ and 
we define the fitness $\psi_\mu$ as a function of the arrival time conditioned
on the disturbance time, the expected fitness $\phi_\mu$ depending only on the arrival time
(averaged over all admissible disturbance times) and, later on, the average fitness $\bar \lambda_\mu$
(averaged over all disturbance and arrival times). We describe the meaning of 
an \textit{evolutionarily stable strategy} (ESS) and
we discuss the easiest cases where either there is no competition ($a=0$) or the probability $p$
of surviving the disturbance is 1,  i.e.~the disturbance has no effect whatsoever (Remark~\ref{rem:a0}).
Section~\ref{sec:mainres} contains our main results. Theorem~\ref{th:implicit} extends
\cite[Appendix A]{cf:IwLe95} and shows that
there is only one possible ESS for the population in response to a fixed disturbance distribution,
which we imagine supported in $[\underline{t}_f,\overline{t}_f]$.
The behaviour of the ESS depends on the value of $(a,p)$. Indeed
there exists a function of $p$, say $a_M=a_M(p)$, such that if $a<a_M$ then the population
starts arriving after a date $x_c>\underline t_f$ and there are arrivals every day until $\overline t_f$; 
if $a=a_M$ then continuous arrivals starts at $\underline t_f$;
if $a>a_M$ then a fraction $\gamma$ of the population arrives at time 0 and the rest arrives
continuously starting from $\underline t_f$.
The dependence on $a$ and $p$ of the fitness of individuals following the ESS is discussed; in particular we show
(see Remark~\ref{rem:maximal}) that the theoretical maximum value of the average fitness $\bar \lambda_\mu$ is not attained by any
ESS. This proves that, even though an ESS is a strategy that each member of the colony considers fair,
it is not the best choice for the colony as a whole. The dependences of every relevant coefficient on
$a$, $p$ and the disturbance distribution are summarized in a table before the beginning of
Section~\ref{sec:uniformdist}.
As an example, we describe the case where the disturbance strikes according to a uniform
distribution (see Section~\ref{sec:uniformdist}).
The effects of climate changes are studied in Section~\ref{sec:climatechange}.
The main questions discussed here are the following. How does the 
ESS change after a climate modification?
What happens if a population would keep the same strategy after a climate change? 
A worse climate, that is a smaller $p$, may change the shape of the ESS delaying the first
arrivals and increasing the fitness of each individual (provided that the population survives the transition).
If the distribution of the disturbance is linearly rescaled, the same happens to the ESS.
The behaviour of the fitness, before the strategy adapts, is studied 
in Propositions~\ref{pro:timeintervals}, \ref{pro:averagefitnessclimatechange}
and \ref{pro:changepaveragefitness}.
A consistent delay of the disturbance
reduces the fitness of all individuals following the former ESS.
If the disturbance arrives earlier than in the past, then
the average fitness of the population increases.
When the competition is weak
then a decrease of $p$ implies a killing of a larger
fraction of the population and hence a lower average fitness.
In Section~\ref{sec:discuss} we discuss and summarize the conclusion of the paper.
Section~\ref{sec:proofs} is devoted to the proofs of our results.

\section{The model}
\label{sec:model}

Iwasa and Levine \cite{cf:IwLe95} studied different ways to model the fact
that individuals choosing an early date of arrival, if no disturbance were present, 
would obtain a larger fitness than those arriving later.
This may be due to the fact that a decrease of reproductive success (better
resources at earlier times, \cite[Case 1]{cf:IwLe95})
or to competition between individuals, for instance those who arrive earlier
may feast on food, while those arriving later will not (\cite[Case 2]{cf:IwLe95}).
We focus on the second case, where Iwasa and Levin proved the emergence of a
mixed strategy for the arrival dates (which is more interesting and realistic than
Case 1, where all individuals choose the same date).

Suppose that $\mu$ is the probability measure, supported on $[0,+\infty)$,
 according to which individuals
choose the date of arrival to the breeding sites. 
Its cumulative distribution function $F_\mu$ is defined,
as usual, by $F_\mu(x):=\mu((-\infty,x])$ for all $x \in \mathbb{R}$.
The population is struck by a single disturbance (e.g.~storm, frost) whose date is randomly
distributed on $(\underline{t}_f,\overline t_f)$ with density $f$.
It may be that $\underline{t}_f=0$ ($\underline{t}_f$ is
the first possible date of disturbance while $\overline{t}_f$ is the last).
Individuals already present when the disturbance occurs, survive with probability $p\in [0,1]$.
The fitness of an individual is a decreasing function of the fraction of individuals
(in the whole population) who are
already present when it arrives, and of the parameter $a\ge 0$ 
which represents the strength of intra-species competition. 
Inspired by \cite{cf:IwLe95} we choose the following expression for
the fitness  $\psi_\mu$ of an individual arriving at date $y$, given
that the disturbance strikes at time $x$:
\[
 \psi_\mu(y|x)=\begin{cases}
	    p\exp(-aF_\mu(y)) & \text{if }0\le y\le x;\\
	    \exp(-apF_\mu(x)-aF_\mu(y)+aF_\mu(x)) & \text{if } y>x.
           \end{cases}
\]
Thus $\psi_\mu(y|x)$ is the fitness 
(under the strategy $\mu$) of an individual arrived at time $y$  
when the disturbance strikes at time $x$.
It is worth noting that we are implicitly exploiting the Law of Large Numbers (LLN),
since the exact value of $\psi_\mu(y|x)$, for instance if $y\le x$, is
\[
 \psi_\mu(y|x)=p\exp(-a N(y)/N),
\]
where $N(y)$ is the random number of individuals arrived before $y$ and $N$ is
the population size. If $N$ is large, the LLN implies that $N(y)/N$ can be
approximated by $F_\mu(y)$. Similarly one proceeds with the case $y>x$.

We consider the expectation of the fitness, with respect to the disturbance date:
the expected fitness of an individual arrived at $y$ is
\[
 \phi_\mu(y)=\int_0^{\overline t_f}\psi_\mu(y|x)f(x)\dd x.
\]
More explicitly, for $y\ge0$,
\begin{equation}\label{eq:phiexplicit}
\begin{split}
 \phi_\mu(y)&=\exp(-aF_\mu(y))\Big[\int_0^y\exp(a(1-p)F_\mu(x))f(x)\dd x+p\int_y^{\overline t_f}f(x)\dd x \Big]\\
 &=\exp(-aF_\mu(y))\Big[\int_0^y (\exp(a(1-p)F_\mu(x))-p)f(x)\dd x+p\Big] \\
 \end{split}
\end{equation}
We note that, in the last integral of equation~\eqref{eq:phiexplicit}
(as well in every integral of a function of type $k(x)f(x)$ in the sequel), one could use $+\infty$ instead of $\overline t_f$ 
and the value of the integral would be the same (as the definition of $\phi_\mu$); moreover
if $y>\overline{t}_f$, by $\int_y^{\overline t_f}$ we mean $-\int^y_{\overline t_f}$.
%

We assume that, the population follows evolutionarily stable strategies,
that is, distributions of arrival times which grant no advantage to any particular
choice of arrival date. More precisely, 
an \textit{evolutionarily stable strategy} (\textit{ESS})  $\mu$, is such that
no mutant can have an advantage, namely 
for all $y\in \mathrm{supp}(\mu)$ and for all $z\in\R^+$, 
$\phi_\mu(y)\ge\phi_\mu(z)$. 
This implies that $\phi_\mu(y)=\lambda$ for all $y\in \mathrm{supp}(\mu)$, where
$\lambda:=\sup\{\phi_\mu(y)\colon y\in\R^+\} \le 1$.
We recall that $\mathrm{supp}(\mu)$ is the (closed) set of $y$ such that $\mu(y-\eps,y+\eps)>0$ 
for all $\eps>0$. We will also need to define
the essential support $\mathrm{Esupp}(f)$ of a real function $f$, which is
the support of the associated measure $A\mapsto\int_Af(x)\dd x$ (for instance,
if $f$ is continuous then $\mathrm{Esupp}(f)=\overline{\{f \not = 0\}}$).
We recall that in \cite{cf:IwLe95} the probability distribution of the disturbance was considered as
absolutely continuous with a single peak density $f$ which was taken, on $[0,\overline{t}_f]$ 
as a polynomial vanishing at the extrema of the interval.
We only assume that
$f$ is a probability density, supported in $[0,\overline{t}_f]$.
From now on, without loss of generality, we assume $\overline t_f:=\max \mathrm{Esupp}(f)$ 
and $\underline t_f:=\min \mathrm{Esupp}(f)$.

%
%
%

We are interested in studying $\mu$ as a function of $a$ and $p$ (and of $f$, but here $f$ is thought as
fixed).
The extremal cases where either $a=0$ or $p=1$ are easy to describe.

\begin{rem}\label{rem:a0}
\begin{enumerate}
 \item If $a=0$ and $p=1$, there is no competition and the disturbance has no effect.
Then $\phi_\mu(y)=1$ for all $y>0$; moreover, every $\mu$ is an ESS (indeed the disturbance has no chance to shape an evolutionary
response of the species).
\item If $a=0$ and $p<1$, there is no competition. From equation~\eqref{eq:phiexplicit} we have,
$ \phi_\mu(y)=1-(1-p)\int_y^{\overline t_f}f(z)\dd z$ which is non decreasing and continuous 
and $\phi_\mu(y)=1$ for all $y \ge \overline t_f$.
We see here that there is no dependence on $\mu$. A probability measure $\mu$ is thus an ESS if and only if 
$\mathrm{supp}(\mu) \subseteq [\,\overline t_f, \infty)$. This means that all idividuals will arrive after the last
possible date of disturbance.
\item If $a>0$ and $p=1$, there is competition and the disturbance has no effect.
From equation~\eqref{eq:phiexplicit} we have $\phi_\mu(y)=\exp(-aF_\mu(y))$ 
which is right-continuous and nonincreasing. Using the same arguments as in the proof of 
Theorem~\ref{th:implicit} 
it is straightforward to prove that there is a unique ESS, namely $\mu=\delta_0$. Everybody arrives at the 
first possible arrival date (indeed there is no risk in doing so).
\item If  $a>0$ and $p=0$, the case has been studied in \cite[Case 2]{cf:IwLe95} and can be retrieved as a
particular case of Theorem~\ref{th:implicit}. There exists a
critical date $x_c$ (depending on $a$ and $f$) after which individuals start arriving according to an absolutely continuous
measure $\mu$ such that $supp(\mu)=(x_c,\overline t_f)$. The expected fitness of each individual
is $\lambda=\frac{1}{1+a}$.
\end{enumerate}
\end{rem}

The interesting case is when $a>0$ and $p<1$, that is, competition in the population and
effective disturbance.
These are the constraints which we assume thereafter.

\section{Main result}
\label{sec:mainres}

We are able to prove (Theorem~\ref{th:implicit}) that given $p<1$, there exists a
critical $a_M$, depending only on $p$ (not on $f$) such that:
\begin{enumerate}
 \item if $a<a_M$ then the ESS is as follows: there is a critical date $x_c>\underline{t}_f$ 
(depending on $a$, $p$ and $f$) after which individuals arrive continuously while
disturbances are possible (see for instance Figure~\ref{fig:impulse02-02-05-09}).
This extends the previously known result for $p=0$ since $a_M(0)=+\infty$ (see \cite[Case 2]{cf:IwLe95});
\item if $a=a_M$ then the ESS is as before, with $x_c=\underline{t}_f$ (individuals arrive throughout the whole
period of possible disturbance, see for instance Figure~\ref{fig:impulse5-02-05-09});
\item if $a>a_M$ then the ESS is such that a fraction $\gamma$ of individuals arrive at 0, and
the remaining arrive continuously during the whole
period of possible disturbance (see for instance Figure~\ref{fig:impulse5-05-05-09}).
\end{enumerate}

Before stating our main result, we define the quantities $a_M$, $x_c$ and $\gamma$.
\begin{defn}
\label{def:aM}
 If $p=0$, then $a_M(p):=+\infty$; if $p\in(0,1)$, then $a_M(p)$ 
is the solution to the equation
\begin{equation}\label{eq:aM}
  \int_1^{\exp(a_M)}\frac{\dd z}{z^{1-p}-p} =\frac{1}{p}.
\end{equation}
\end{defn}
Note that the solution to equation \eqref{eq:aM} exists and is unique 
since the l.h.s.~is a continuous, strictly increasing function of $a_M$
which vanishes at $a_M=0$ and goes to infinity as $a_M\to \infty$. The inverse function of $a_M$ 
will be denoted by $p_M(a)$ (by definition $p_M(+\infty):=0$). Clearly $a \le a_M(p)$ (resp.~$a \ge a_M(p)$)
if and only if $p \le p_M(a)$ (resp.~$p\ge p_M(a)$).

\begin{defn}
If $a > a_M(p)$ let $x_c=x_c(a,p,f):=0$. If $a \le a_M(p)$ 
define $x_c$ as the maximal solution to the equation
\begin{equation}\label{eq:xc}
 \int_{x_c}^{\overline t_f}f(x)\dd x=\Big( 1-p+\Big( \int_1^{\exp(a)}\frac{\dd z}{z^{1-p}-p} 
\Big)^{-1}\Big)^{-1}.
\end{equation}
\end{defn}
Note that the solution to equation \eqref{eq:xc} exists and is unique, 
since the r.h.s.~is positive and strictly smaller than 1, while the l.h.s.~is a continuous, nonincreasing 
function of $x_c$ which takes values $1$ and $0$ at $x_c=0$ and $x_c=\overline t_f$ respectively. 
By definition if $a \le a_M(p)$, $x_c \in \mathrm{Esupp}(f)$.

\begin{defn}
 \label{def:gamma}
 If $a \le a_M(p)$ let $\gamma=\gamma(a,p):=0$. If $a > a_M(p)$ let $\gamma$ be
the unique solution to the equation
\begin{equation}\label{eq:gamma}
\exp(-a\gamma)\int_{\exp(a\gamma)}^{\exp(a)}\frac{\dd z}{z^{1-p}-p}=\frac1p.
\end{equation}
\end{defn}
The solution exists and is unique since the l.h.s.~is a continuous, strictly decreasing
function of $\gamma$ which takes values in $(1/p, +\infty)$ when $\gamma=0$ and is equal to $0$ 
when $\gamma=1$. Observe that $\gamma(a,0)=0$ for all $a>0$.

Note that, in accordance with \cite{cf:IwLe95}, if $p=0$ then $a_M=+\infty$, $\gamma=0$, and 
$\int_{x_c}^{\overline t_f}f(x)\dd x=a/(1+a)$. Let us discuss some general properties of 
$a_M$, ${x_c}$ and ${\gamma}$. For details on the proofs,
see Section~\ref{sec:proofs}.
\begin{rem}
\label{rem:properties}
\begin{enumerate}
 \item The map $p\mapsto a_M(p)$  is continuous and
strictly decreasing; see Figure~\ref{fig:aM}. 
By using  elementary techniques of implicitly defined functions it is not
difficult to prove that 
  $ \lim_{p\to0^+} a_M(p)=+\infty$, 
$\lim_{p\to1^-}a_M(p)=0$.
\item 
Given a fixed $f$, 
we have that $p \mapsto x_c$ and $a\mapsto x_c$  are 
strictly decreasing and left continuous everywhere. The map
$p \mapsto x_c$ is right continuous at $p_0$ 
only in the following cases:
(a) $a>a_M(p_0)$; 
(b) $a<a_M(p_0)$ and $(x_c(a,p_0,f)-\varepsilon,x_c(a,p_0,f)] \subseteq \mathrm{Esupp}(f)$ 
for some $\varepsilon >0$; 
(c) $a=a_M(p_0)$ and
$\underline t_f =0$.
Similarly $a \mapsto x_c$
is right continuous at $a_0$ 
only in the cases:
(a) $a_0>a_M(p)$;
(b) $a_0<a_M(p)$ and  $(x_c(a_0,p,f)-\varepsilon,x_c(a_0,p,f)] \subseteq \mathrm{Esupp}(f)$ 
for some $\varepsilon >0$;
(c) $a_0=a_M(p)$ 
and $\underline t_f=0$.
Again using standard techniques of implicitly defined functions we have 
\begin{equation}\label{eq:xclimits}
 \begin{split}
  \lim_{p\to0^+} x_c(a,p,f) &\le x_c(a,0,f)<\overline t_f= \lim_{a\to0^+} x_c(a,p,f);\\
  \lim_{p\to p_M(a)^-}x_c(a,p,f)&=\underline t_f=\lim_{a\to a_M(p)^-}x_c(a,p,f).\\
\end{split}
\end{equation}
\item
The function $(a,p)\mapsto \gamma(a,p)$ is continuous on $(0,+\infty)\times[0,1)$ and 
$p \mapsto \gamma(a,p)$  is strictly increasing (for all fixed $a \in (0,+\infty)$).
Moreover $\gamma(a,p) <p$ for all $p\in (0,1)$ and for all $a > 0$; 
$a \mapsto \gamma(a,p)$ is strictly increasing in $[a_M(p),+\infty)$ for all fixed $p \in (0,1)$ (see Remark~\ref{rem:lambdadecreasing}). 
%
%
See Figures~\ref{fig:gammap} and \ref{fig:gammaa} for some plots of
$p \mapsto \gamma(a,p)$
and $a \mapsto \gamma(p,a)$.
Moreover, we have
\[   
\lim_{p\to0^+} \gamma(a,p)=0,\qquad \lim_{p\to 1^{-}} \gamma(a,p)=1, \qquad 
   \lim_{a \to \infty} \gamma(a,p)=p,
\]
where the last limit holds for all $p>0$.
\end{enumerate}
\end{rem}

%
%
\begin{figure}[H]
  \centering
    \includegraphics[width=0.35\textwidth]{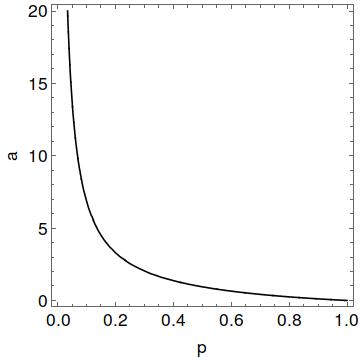}
  \caption{
$p \mapsto a_M(p)$.}\label{fig:aM}
\end{figure}

 \begin{figure}[ht!]
 \begin{minipage}{0.5\textwidth}
 \centering
  \includegraphics[width=0.75\textwidth]{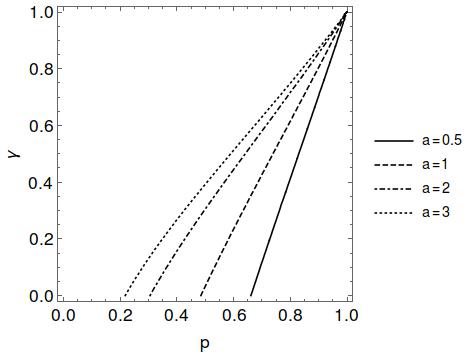}
  \caption{$p \mapsto \gamma(p,a)$; $a=3,2,1,0.5$.}\label{fig:gammap}
 \end{minipage}
\begin{minipage}{0.5\textwidth}
\centering
 \includegraphics[width=0.75\textwidth]{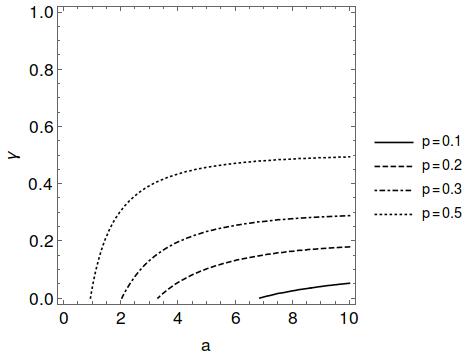}
  \caption{$a \mapsto \gamma(p,a)$; $p=0.5,0.3,0.2,0.1$.}\label{fig:gammaa}
\end{minipage}
\end{figure}


\begin{teo}\label{th:implicit}
Let $a_M$, $x_c$ and $\gamma$ be as previously defined.
 Given $a>0$, $p\in[0,1)$ and $f$ a probability density 
supported in $[0,\overline t_f])$ (where by definition $\overline t_f=\max \mathrm{Esupp}(f)$),
there exists a unique ESS $\mu$.
 In particular, $\mu=\gamma\delta_0+(1-\gamma)\nu$, where $\nu$ is an absolutely
 continuous probability measure with  
$\mathrm{supp}(\nu)=[x_c,\overline t_f]\cap\mathrm{Esupp}(f)$.
 The cumulative distribution function $F_\nu$ is implicitly defined, on $x \ge x_c$, by
 \begin{equation}\label{eq:implicit}
  \int_{\exp(a\gamma)}^{\exp(a(1-\gamma)F_\nu(x)+a\gamma)}
\frac{\dd z}{z^{1-p}-p}=\frac{1}{\lambda}\int_{x_c}^xf(y)\dd y.
 \end{equation}
 The value of $\lambda$ (the supremum of $\phi_\mu$) is
 \begin{equation}\label{eq:deflambda}
  \lambda=\lambda(a,p)=
 \begin{cases}
 \displaystyle \frac{1}{1+(1-p)\int_1^{\exp(a)}\frac{\dd z}{z^{1-p}-p}} & \text{if } a \le a_M\\
p\exp(-a\gamma) & \text{if } a > a_M.
 \end{cases}
 \end{equation}
 Finally, $a_M$ is a phase transition value for the competition in the sense that
\[
 \begin{split}
   a<a_M\quad &\Rightarrow\quad \gamma=0,\ x_c>\underline t_f;\\
   a=a_M\quad &\Rightarrow\quad \gamma=0,\ x_c=\underline t_f;\\
   a>a_M\quad &\Rightarrow\quad \gamma \in (0,(1-a_M/a) \wedge p), 
   \ x_c=0.
 \end{split}
\]
\end{teo}
It is worth noting that, in equation~\eqref{eq:implicit}, we can write 
 $F_\mu(x)$ instead of $\gamma+(1-\gamma)F_\nu(x)$.

Several features of the ESS $\mu$ can be read from Theorem~\ref{th:implicit}.
Beside arrivals at 0 (possible only when competition is sufficiently strong), 
individuals arrive only at possible disturbance dates. This means that if it is
certain that during $[s,t]\subsetneq[\underline{t}_f,\overline{t}_f]$, no disturbance
is possible (i.e.~$f(x)=0$ for all $x\in[s,t]$), then the probability of arrivals
in such interval is zero. This is due to competition, which advantages early birds:
there is no point in choosing to arrive in between $s$ and $t$, since there is no
risk in choosing $s$ instead.

In most real cases, one may assume that $f>0$ on $[0,\overline{t}_f]$. In that case
individuals will arrive either (1) avoiding the first part of the possible time period
(weak competition, see Figure~\ref{fig:impulse05-03}) or (2) during the whole interval
$[0,\overline{t}_f]$ without massive
arrivals at 0 (critical competition, Figure~\ref{fig:impulse03-03})
or (3) during the whole interval with massive
arrivals at 0 (supercritical competition, Figure~\ref{fig:impulse01-03}).

The first arrival time $x_c$ is strictly decreasing with respect to $p$ and to $a$ (Remark~\ref{rem:properties});
this means that strong competition and/or high probability of surviving the disturbance, push arrivals to 0.
If $p$ is fixed, competition needs to be above the threshold $a_M(p)$, in order to have arrivals at 0.
On the other hand, if competition is fixed, only weak disturbances lead to early arrivals.



As for $\gamma$, the fraction of arrivals at 0, we know that it increases with $p$.
If $p>0$, from equation~\eqref{eq:implicit}, using $\gamma(a,p) \uparrow p$ as $a \to \infty$, 
it follows that $F_\nu(x) \to 1$ for all $x \in (\underline t_f,\overline t_f]$ as $a \to \infty$. This means that, as the competition increases, 
given that an individual does not arrive at time $0$ (this probability converges monotonically to $1-p$ from above) then
the probability of arriving after $x >\underline t_f$ goes to $0$, that is, in the limit as $a \to \infty$ the arrival distribution
converges to $p \delta_0+(1-p) \delta_{\underline t_f}$
(that is, $\delta_0$ if $\underline t_f=0$). 
Note that this does not happen when $p=0$: in that case $F_\nu(x)=\frac{1+a}a \int_{x_c}^x f(y) \dd y$
which implies that $F_\nu$ converges to the cumulative distribution of the disturbance arrival time.

Associated to a given strategy $\mu$, there is the average fitness $\bar \lambda_\mu:=\int_0^{+\infty} \phi_\mu(y) \mu(\dd y)$.
In particular when $\mu$ is an ESS, then $\phi_\mu(y)=\lambda$ for all $y \in \mathrm{supp}(\mu)$ (where $\lambda=\lambda(a,p)$,
see \eqref{eq:deflambda}), hence $\bar \lambda_\mu=\lambda(a,p)$.
If $a \le a_M$ we can relate $\lambda(a,p)$ to $x_c$: by equations~\eqref{eq:xc} and \eqref{eq:deflambda} 
we have
\begin{equation}\label{eq:implicitmu}
\int_{x_c}^{\overline t_f}f(y)\dd y=\frac{1-\lambda(a,p)}{1-p}= \lambda(a,p) \int_1^{\exp(a)}\frac{\dd z}{z^{1-p}-p}.
\end{equation}
It is not difficult to prove that $(a,p) \mapsto \lambda(a,p)$ is continuous in $(0, +\infty) \times [0,1)$; moreover
$a \mapsto \lambda(a,p)$ and $p \mapsto \lambda(a,p)$ are strictly decreasing functions (see Remark~\ref{rem:lambdadecreasing}) such that 
\[
 \begin{split}
   \lim_{a\to0^+} \lambda(a,p)=1,\qquad &\lim_{a\to +\infty} \lambda(a,p)=0;\\
   \lim_{p\to0^+} \lambda(a,p)=\frac{1}{1+a},\qquad &\lim_{p\to 1^{-}} \lambda(a,p)=\exp(-a), \qquad 
   \end{split}
\]
hence $\lambda(a,0)=(1+a)^{-1} \ge \lambda(a,p) \ge \exp(-a)=\lambda(a,1)$ for all $p \in [0,1]$; moreover $\lambda(a_M(p),p)=p$. 

Thus, when $\mu$ is an ESS 
the average fitness $\bar \lambda_\mu=\lambda(a,p)$ is decreasing with respect to $p$ (when $a>0$), hence,
if we think of the fitness as a probability of survival,
the average rate of survivors is decreasing if the chance of surviving the catastrophe $p$ is increasing.
From a biological point of view, this model suggests that, in the presence of competition, if the disturbance is weaker (that is,
$p$ increases) then the ESS pushes the colony towards an early arrival on the site, increasing the negative effects of the competition
on the fitness (which overcome the positive effects of the weaker disturbance). Hence the stronger the disturbance
the higher the average fitness corresponding to the ESS (in some sense, if we think of the average fitness of the population as
its ``strength'', a strong disturbance will select a stronger population).

 \begin{figure}[H]
 \begin{minipage}{0.33\textwidth}
 \centering
  \includegraphics[width=0.85\textwidth]{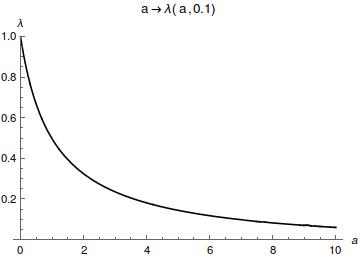}
 \end{minipage}
\begin{minipage}{0.33\textwidth}
\centering
   \includegraphics[width=0.85\textwidth]{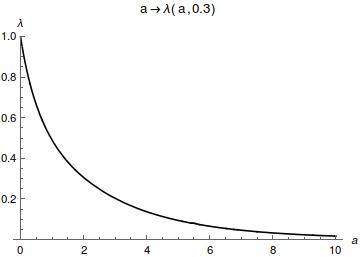}
\end{minipage}
\begin{minipage}{0.33\textwidth}
\centering
  \includegraphics[width=0.85\textwidth]{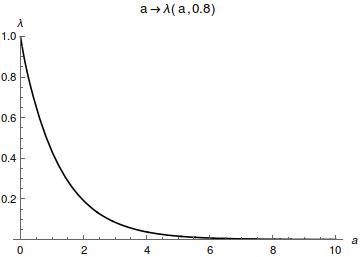}
\end{minipage} 
\\ \vskip 0.5 truecm
%
 \begin{minipage}{0.33\textwidth}
 \centering
  \includegraphics[width=0.85\textwidth]{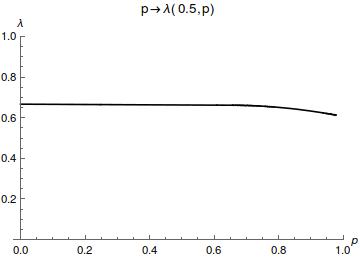}
 \end{minipage}
\begin{minipage}{0.33\textwidth}
\centering
  \includegraphics[width=0.85\textwidth]{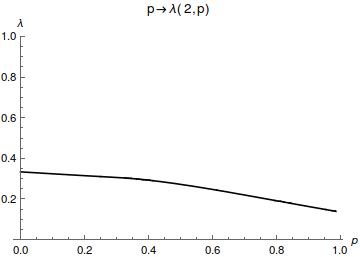}
\end{minipage}
\begin{minipage}{0.33\textwidth}
\centering
  \includegraphics[width=0.85\textwidth]{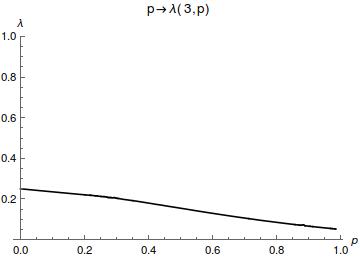}
\end{minipage}
\end{figure}

One may wonder if there are strategies $\mu$ which (given the environment $(a,p,f)$) lead to a larger $\bar \lambda_\mu$,
and whether these strategies are Evolutionary Stable Strategies.

\begin{rem}\label{rem:maximal}
The supremum of the map
 $\mu \mapsto \int \phi_\mu(y) \mu(\dd y)=\bar \lambda_\mu$ 
  is
 $(1-\exp(-a))/a$ and is attained by $\mu$ if and only
if $\inf \mathrm{supp}(\mu) \ge \sup \mathrm{Esupp}(f)=\overline t_f$ and $F_\mu$ is continuous
(and this supremum holds for any fixed $p$, see details in Section~\ref{sec:proofs}).
This means that the best choice for the population as a whole, is to start arriving after the last possible
date of disturbance. In particular no ESS can attain
this maximum value (an ESS is supported in $[0,\overline{t}_f)$).
An ESS is a strategy which is well accepted by every individual (since the fitness is constant and maximal),
while the maximizing strategies imply that some individuals
``accept'' a lower fitness for the benefit of the whole colony (note that the continuity of $F_\mu$ implies
that the population cannot choose to arrive simultaneously at $\overline{t}_f$, which would guarantee the same fitness
for everyone).
Moreover, even if the maximal fitness is attained only arriving after $\overline{t}_f$, 
it is possible to prove that we can find a sequence of strategies (not ESSs) $\{\mu_n\}_{n \ge 1}$, 
supported in $[0,\overline t_f]$ such that $\bar \lambda_{\mu_n}> (1-\exp(-a))/a-1/n$.
In Figure~\ref{fig:averagedfitness} the dashed line represents $\sup_\mu\bar\lambda_\mu$,
whereas the dotted  and the solid ones are the average fitness when the population follows the ESS,
$\lambda(a,1)$ and $\lambda(a,0)$, respectively. The filled region represents all 
possible values for $\lambda(a,p)$ with $p\in(0,1)$. Note that following the ESS a population cannot
achieve the maximal average fitness, nevertheless if $a$ is either small or large, 
then the dashed and the solid line are close, hence, if $p=0$ or at least $p$ is small,
then the average fitness is not so far from its theoretical supremum.
\end{rem}

\begin{figure}[H]
  \centering
    \includegraphics[width=0.35\textwidth]{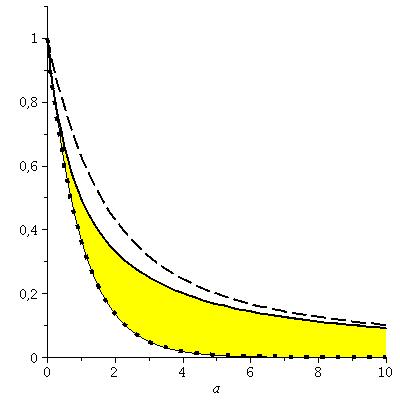}
  \caption{Maximum average fitness and the admissible region for $\lambda(a,p)$.}\label{fig:averagedfitness}
\end{figure}

In the following table, we summarize the main properties of the coefficients $a_M$, $\gamma$, $\lambda$ and $x_c$;
by $\nearrow\, x$ and $\searrow\, x$ we mean that a particular coefficient is increasing or decreasing with respect 
to the parameter $x$.
\begin{center}
\begin{tabular}{|c|c|c|} \hline
Coefficients & Dependence & Properties \\ \hline
 \tiny{Phase transition competition $a_M$} & $p$ & $\searrow\, p$\\
& & $a_M \in [0,+\infty]$ \\ \hline
\tiny{Probability of emigration at time $0$, $\gamma$} & $a$, $p$ & $\nearrow\, a$ $\ \ $ $\nearrow\, p$ \\ 
 & & $\gamma=0$ if $a \le a_M$ \\
& & $\gamma \in (0,p)$ if $a > a_M$\\ \hline
\tiny{Fitness of individuals for the ESS, $\lambda$} &  $a$, $p$ &  $\searrow\, a$ $\ \ $ $\searrow\, p$\\
 & & $\exp(-a) \le \lambda(a, \cdot) \le (1+a)^{-1}$\\ \hline
\tiny{Maximum average fitness $\bar\lambda$ for a generic strategy} & $a$ & $\bar\lambda(a,p)=(1-\exp(-a))/a$ \\ \hline
\tiny{Earliest arrival time for an ESS $x_c$} &  $a$, $p$, $f$ &  $\searrow\, a$ $\ \ $  $\searrow\, p$\\
 & & $x_c \in [\underline t_f, \overline t_f]$ if $a \le a_M$ \\
& & $x_c=0$ if $a > a_M$
\\ \hline
\end{tabular}
\end{center}

Given $(a,p,f)$, Theorem~\ref{th:implicit} gives the unique ESS $\mu$ and its average fitness $\lambda$.
We note that the map $(a,p,f) \mapsto (\lambda, \mu)$ is not injective. Indeed, 
at least when $a\le a_M(p)$, it might be that $\lambda(a,p)=\lambda(b,q)$ for some $q$ and $b$.
Even if we fix $(a,p)$, at least in the subcritical case,
different disturbance distributions may lead to the same ESS. 
Indeed, suppose that 
$a<a_M(p)$ then $x_c>0$ and $\gamma=0$. Given $\mu$ is fixed, from equation~\eqref{eq:implicit},
$f$ is uniquely determined on $[x_c, \overline t_f
]$.
Nevertheless, on $[0,x_c]$ the only constraint is $\int_0^{x_c} f(x) \dd x=(p-\lambda)/(1-p)$,
which can be satisfied by infinitely many distributions.
This means that different levels of competition, and/or of climate, can lead to the same response $\mu$ and same fitness
$\lambda$.

\subsection{Uniformly distributed disturbances}
\label{sec:uniformdist}

In this example we suppose that $f(t) = (\overline t_f-\underline t_f)^{-1}\ident_{[\, \underline t_f, \overline t_f]}(t)$, that is,
the law of the disturbance is uniformly distributed in the interval $[\underline t_f, \overline t_f]$
(we also write $f\sim\mathcal U(\underline t_f, \overline t_f)$). 
In this case some explicit computations are possible.

The coefficients $a_M$, $\lambda$, $\gamma$ do not depend of $f$; the only coefficient depending on $f$ 
is the first time of arrival $x_c$ (which is nonzero only if $a\le a_M$). If $a\le a_M$, then
\[
 x_c=\underline t_f +\alpha(a,p)(\overline t_f-\underline t_f)
\]
where $\alpha(a,p):=\frac{1-pC(a,p)}{1+(1-p)C(a,p)}$, with $C(a,p)=\int_1^{\exp(a)} \frac{\dd z}{z^{1-p}-p} \le 1/p$.
This means that, in the subcritical case, the ratio $R$ between the length of the arrival times interval and that of the disturbance
times interval, depends only on $a$ and $p$ (not on $f$); 
indeed $R:=(\overline t_f - x_c)/(\overline t_f - \underline t_f)=(1-p+C(a,p)^{-1})^{-1}$.

The cumulative distribution function $F_\mu$ can be computed using equation~\eqref{eq:implicit}.
Although when $p\neq0$ no explicit evaluations are possible,
one can see that $F_\mu$ is a rescaling of the cumulative distribution function of the ESS obtained when
the disturbance is uniformly distributed on the interval [0,1] (with the same parameters $a$ and $p$),
which we denote by $F_0^{a,p}$.
More precisely, if $F_\nu=F^{a,p}_\nu$ is the cumulative distribution function 
of the ESS when $f\sim \mathcal U(\underline t_f, \overline t_f)$, then
\begin{equation}\label{eq:rescaling}
 F_\nu^{a,p}(x)=F_0^{a,p} \Big ( \frac{x- \underline t_f \vee x_c}{\overline t_f- \underline t_f \vee x_c}
\Big )
\end{equation}
(recall that $\underline t_f \vee x_c=x_c$ if $ a< a_M$ while $\underline t_f \vee x_c=\underline t_f$ if $ a \ge a_M$). 
Using computer-aided numerical solutions, in Figures~\ref{fig:impulse02-02-05-09}--\ref{fig:impulse5-05-03-1}
we plot the cumulative distribution functions of the ESS (solid line) and  of the disturbance (dashed),
with different parameters and disturbance intervals.
In the first row, Figures~\ref{fig:impulse02-02-05-09}--\ref{fig:impulse5-05-05-09}, we take $f\sim\mathcal U(0.5,0.9)$
while in the second row, Figures~\ref{fig:impulse02-02-03-1}--\ref{fig:impulse5-05-03-1}, we have  
$f\sim\mathcal U(0.3,1)$.

In the three figures in each row, the parameters $(a,p)$ are, respectively,
$(0.2,0.2)$, $(5,0.2)$ and $(5,0.5)$. 
This implies that the first figure represents a subcritical case ($a=0.2>a_M(0.2)=3.30447$).
Note that in both rows in this case the arrivals start towards
the end of the disturbance intervals and the ratio between the length of the two intervals representing
the arrival times and the disturbance time is the same, $R=0.208552$. 

The second figure of the row is a supercritical case ($a=5>a_M(0.2)$), where a fraction $\gamma=0.10142$ of the
population arrives at time 0. In both rows the fraction $\gamma$is the same ($\gamma$ depends only on $a$ and $p$).

The third figure is again a supercritical case ($a=5>a_M(0.5)=0.941046$), where strong competition
forces a larger fraction of the population ($\gamma=0.456433$) to arrive early.
Along with the values of $a$ and $p$, in Figures~\ref{fig:impulse02-02-05-09}--\ref{fig:impulse5-05-03-1} 
we write the explicit value of the maximum average
fitness $\lambda$. 
Note that $\lambda$ increases when either $p$ decreases or $a$ increases, 
the maximum competition $a_M$.

 \begin{figure}[H]
 \begin{minipage}{0.33\textwidth}
 \centering
  \includegraphics[width=0.95\textwidth]{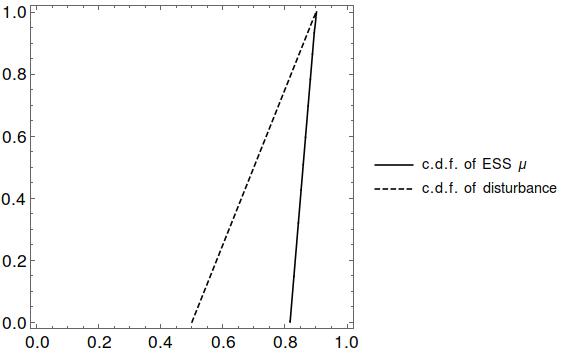}
 \caption{ \tiny{
  $a=0.2$, $p=0.2$, 
  $\lambda=0.833158$. 
  } }\label{fig:impulse02-02-05-09}
 \end{minipage}
\begin{minipage}{0.33\textwidth}
\centering
 \includegraphics[width=0.95\textwidth]{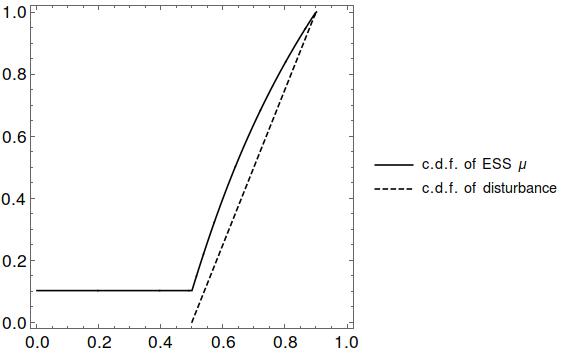}
  \caption{\tiny{$a=5$, $p=0.2$, $\lambda=0.120448$. 
  }}\label{fig:impulse5-02-05-09}
\end{minipage}
\begin{minipage}{0.33\textwidth}
\centering
 \includegraphics[width=0.95\textwidth]{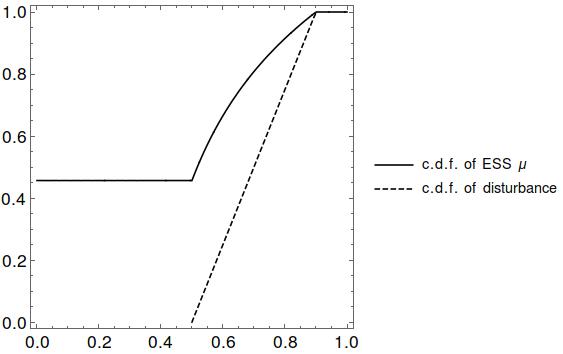}
  \caption{\tiny{$a=5$, $p=0.5$, $\lambda=0.0510315$. 
  }}\label{fig:impulse5-05-05-09}
\end{minipage} \\ \vskip 0.5 truecm
%
 \begin{minipage}{0.33\textwidth}
 \centering
  \includegraphics[width=0.95\textwidth]{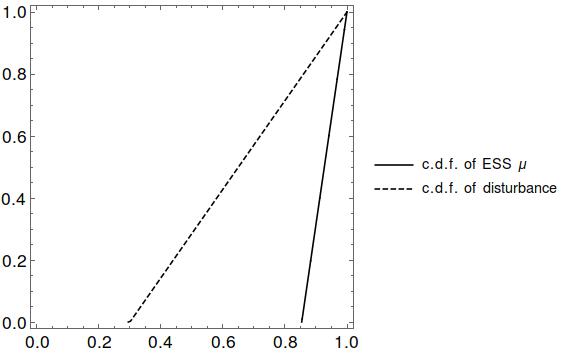}
  \caption{\tiny{$a=0.2$, $p=0.2$, $\lambda=0.833158$. 
  }}\label{fig:impulse02-02-03-1}
 \end{minipage}
\begin{minipage}{0.33\textwidth}
\centering
 \includegraphics[width=0.95\textwidth]{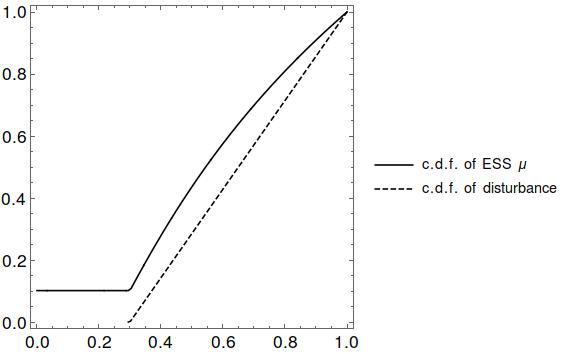}
  \caption{\tiny{$a=5$, $p=0.2$, $\lambda=0.120448$. 
  }}\label{fig:impulse5-02-03-1}
\end{minipage}
\begin{minipage}{0.33\textwidth}
\centering
 \includegraphics[width=0.95\textwidth]{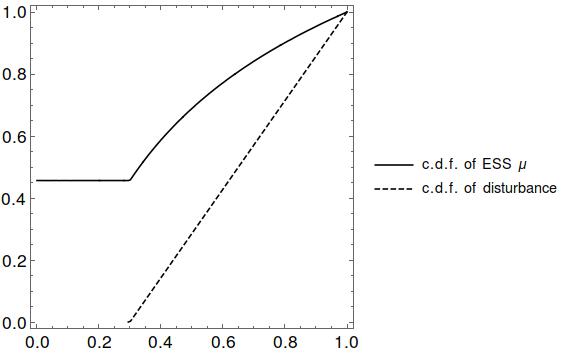}
  \caption{\tiny{$a=5$, $p=0.5$, $\lambda=0.0510315$. 
  }}\label{fig:impulse5-05-03-1}
\end{minipage}
\end{figure}


As we have seen in Figures~\ref{fig:impulse02-02-05-09}--\ref{fig:impulse5-05-03-1}, by \eqref{eq:rescaling},
it is enough to study the case $f\sim\mathcal U(0,1)$.
In 
Figures~\ref{fig:impulse01-01}-\ref{fig:impulse05-05} we plot the cumulative distribution function $F_0^{a,p}$ (solid line)
corresponding to  nine couples
$(p,a)$, together with the cumulative distribution function of the disturbance (dashed line).
In each row $p$ takes the same value ($p=0.1,\ 0.3$ and $0.5$ in the first, second and third row respectively),
while $a$ is constant along each column ($a=a_M(0.1),\ a_M(0.3)$ and $a_M(0.5)$  in the first, second and third column respectively).
In this way, figures on the diagonal represent critical cases, figures on the upper triangle are subcritical cases
and figures on the lower triangle are supercritical cases.
 \begin{figure}[H]
 \begin{minipage}{0.33\textwidth}
 \centering
  \includegraphics[width=0.95\textwidth]{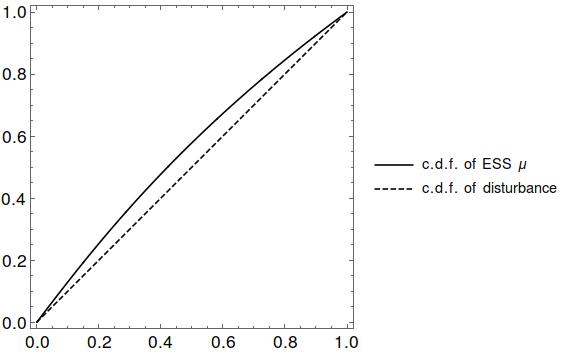}
 \caption{ \tiny{
  $a=a_M(0.1)=6.893$, $p=0.1$, $\lambda=0.0995$, $\gamma=0$.} }\label{fig:impulse01-01}
 \end{minipage}
\begin{minipage}{0.33\textwidth}
\centering
 \includegraphics[width=0.95\textwidth]{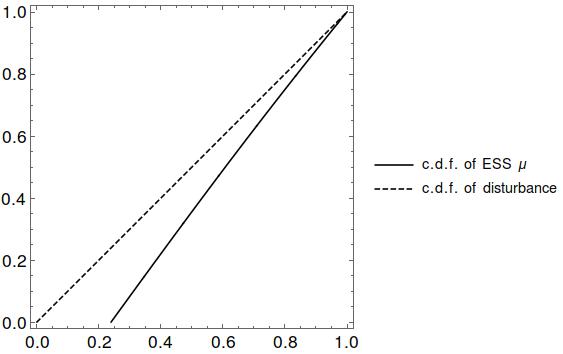}
  \caption{\tiny{$a=a_M(0.3)=2.075$, $p=0.1$, $\lambda=0.3153$, $\gamma=0$.}}\label{fig:impulse03-01}
\end{minipage}
\begin{minipage}{0.33\textwidth}
\centering
 \includegraphics[width=0.95\textwidth]{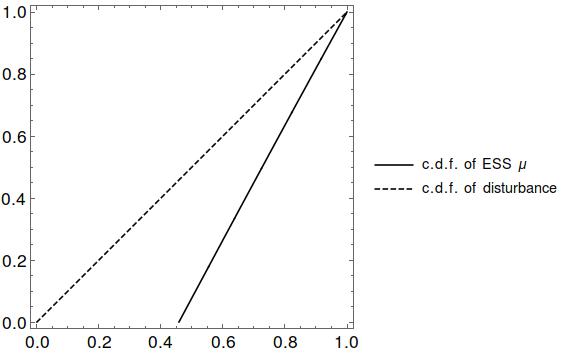}
  \caption{\tiny{$a=a_M(0.5)=0.941$, $p=0.1$,  $\lambda=0.5123$, $\gamma=0$.}}\label{fig:impulse05-01}
\end{minipage} \\ \vskip 0.5 truecm

 \begin{minipage}{0.33\textwidth}
 \centering
  \includegraphics[width=0.95\textwidth]{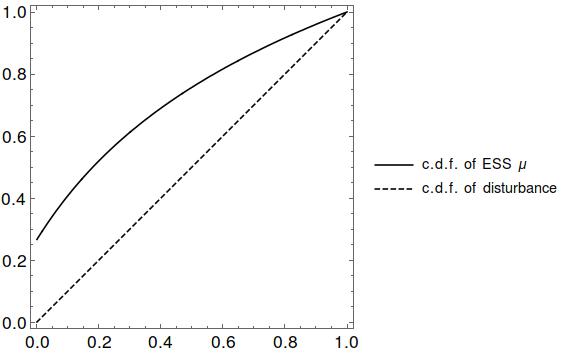}
 \caption{ \tiny{
  $a=a_M(0.1)$, $p=0.3$, $\lambda=0.0478$, $\gamma=0.2663$.} }\label{fig:impulse01-03}
 \end{minipage}
\begin{minipage}{0.33\textwidth}
\centering
 \includegraphics[width=0.95\textwidth]{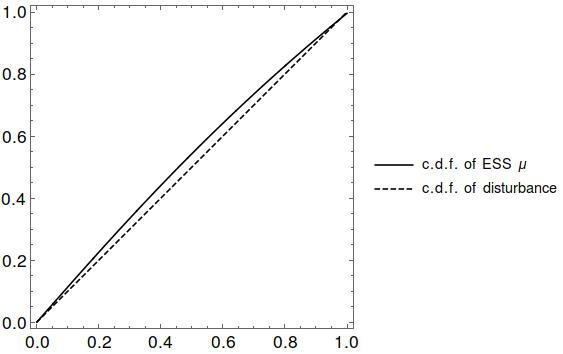}
  \caption{\tiny{$a=a_M(0.3)$, $p=0.3$, $\lambda=0.296271$, $\gamma=0$.}}\label{fig:impulse03-03}
\end{minipage}
\begin{minipage}{0.33\textwidth}
\centering
 \includegraphics[width=0.95\textwidth]{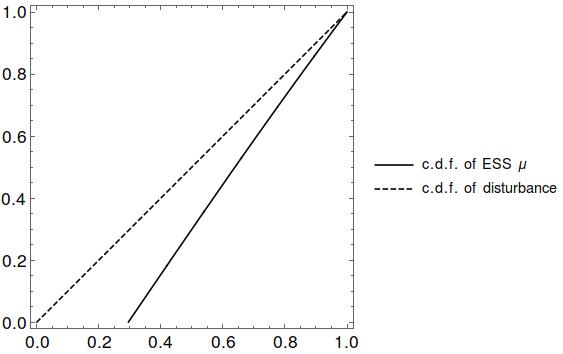}
  \caption{\tiny{$a=_M(0.5)$, $p=0.3$, $\lambda=0.5066$, $\gamma=0$.}}\label{fig:impulse05-03}
\end{minipage} \\ \vskip 0.5 truecm

 \begin{minipage}{0.33\textwidth}
 \centering
  \includegraphics[width=0.95\textwidth]{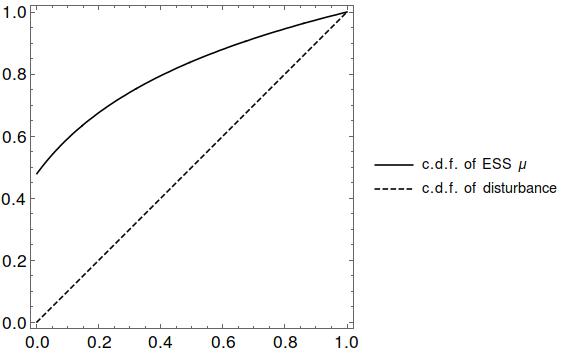}
  \caption{\tiny{$a=a_M(0.1)$, $p=0.5$,  $\lambda=0.0184$, $\gamma=0.4788$.}}\label{fig:impulse01-05}
 \end{minipage}
\begin{minipage}{0.33\textwidth}
\centering
 \includegraphics[width=0.95\textwidth]{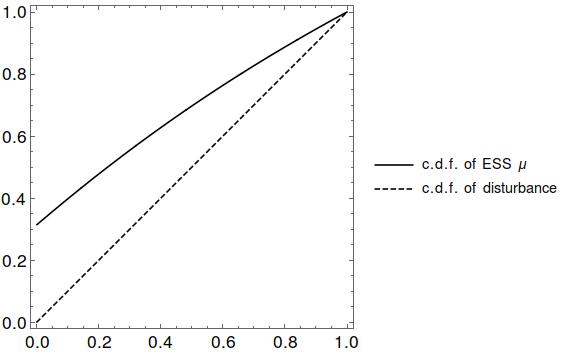}
  \caption{\tiny{$a=a_M(0.3)$, $p=0.5$, $\lambda=0.2606$, $\gamma=0.3141$.}}\label{fig:impulse03-05}
\end{minipage}
\begin{minipage}{0.33\textwidth}
\centering
 \includegraphics[width=0.95\textwidth]{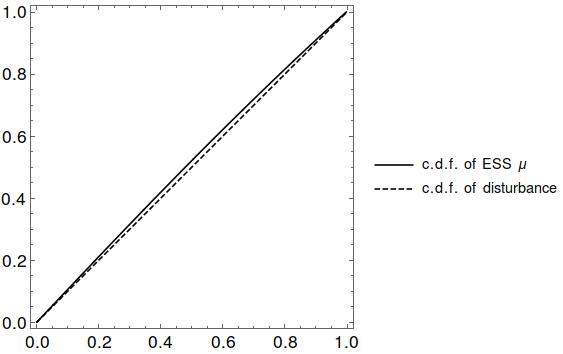}
  \caption{\tiny{$a=a_M(0.5)$, $p=0.5$, $\lambda=0.5011$, $\gamma=0$.}}\label{fig:impulse05-05}
\end{minipage}
\end{figure}

\subsection{Climate changes}
\label{sec:climatechange}


In our model the climate can be represented by the couple $(p,f)$, that is, the distribution and the strength of the
disturbance; more precisely the lower $p$ the stronger the disturbance. 
One can argue that also the competition parameter $a$ could be affected by the climate, nevertheless, in this paper,
we prefer to think of it as a characteristic of the population.

When climate changes (i.e.~ the couple $(p,f)$ changes), there are two interesting questions.
\begin{enumerate}
 \item How does the ESS change reflects a climate change? In other words, can we predict in which respect the new ESS
 will differ from the previous one? 
\item What happens if a colony keeps the same strategy of arrivals after a climate change? Does the average fitness of the
population decrease or increase?
\end{enumerate}

%
%
Some answers can be obtained if we imagine that the climate change affects just one of the parameters $p$ and $f$.

The answer to question 1, is given by Theorem~\ref{th:implicit}.
If $p$ decreases, then $\lambda(a,p)$ increases: after the population has adapted to the new climate 
the common fitness will have increased (of course supposing that the population is sufficiently large to survive the transition period).
Moreover, a decrease of $p$ may lead from an ESS with massive arrival at 0, to an ESS with arrivals after a date 
$x_c>\underline{t}_f$.
Let us discuss now what happens if $f$ changes, that is if it moves from $f_1$ to $f_2$. 
The fitness of every individual, according to an ESS adapted to $f_i$, is
$\lambda(a,p)$ which does not depend on $f_i$, thus that will not change in case of rapid adaptation (question 1).
What will change is the distribution of the arrival times and possibly its support.
In the case of uniformly distributed disturbances, the answer to question 1 is given by equation \eqref{eq:rescaling}:
a change of the interval during which disturbances occur is simply reflected by a rescaling of the ESS.
It is worth noting that an analogous result holds if we take a generic density $f$ defined on $[0,1]$ and we rescale it as
follows
\begin{equation}\label{eq:rescaling1}
 f_i(t):=f \Big ( \frac{t-\underline t_f}{\overline t_f-\underline t_f}\Big );
\end{equation}
the effect on the ESS in this general case is still the rescaling given by equation~\eqref{eq:rescaling}.


The second question is particularly interesting in the case of rapid climate changes, since the adaptation of the population
(moving from the old ESS to the new one) could require several generations; thus the changes may endanger the 
survival of the population.
A sudden change of climate may be an advantage for some
individuals (i.e.~some arrival dates) and a disadvantage for other individuals.
Even in the case of a simple anticipation of the disturbance the situation is not trivial.
Suppose that $f_1$ is supported in $[\underline t_1, \overline t_1]$
and $f_2$ is supported in $[\underline t_2, \overline t_2]$, with $\overline t_2< \overline t_1$.
Individuals arriving at $y>\overline t_2$ are sure that the disturbance is over but they might have more 
competitors still alive (for instance those arriving at $y-\varepsilon$), thus it is not clear
whether the climate change implies a larger or smaller fitness.
The following proposition gives a partial answer: let us denote by $\phi_\mu^{(i)}(y)$ the fitness of an
individual, in a population following the strategy $\mu$ (not necessarily an ESS), 
which chooses the arrival date $y$ and has to face
the disturbance which is regulated by $f_i$. If $f_2$ is a delay of the disturbance, that is, in the second scenario
disturbances strike later in the season, early birds suffer a decrease of their fitness.
The second part of the proposition tells us that, at least for uniform disturbances, if $f_2$ delays
the beginning of the disturbance but also anticipates the end of the disturbance season, then later birds
have an advantage.


\begin{pro}\label{pro:timeintervals}
 Consider a distribution $\mu$ and two densities $f_1$ and $f_2$ on the  intervals 
 $[\underline t_1, \overline t_1]$, $[\underline t_2, \overline t_2]$.
 \begin{enumerate}
 \item If $\underline t_1 < \underline t_2$ then 
 $\phi_\mu^{(1)}(y) \ge \phi_\mu^{(2)}(y)$ for all $y \in [\underline t_1, \underline t_2]$.
 \item 
 If $\mu$ is the ESS associated to $f_1$, $D_2$ is a random variable with density $f_2$ and
 $\Prob(\underline t_2<D_2<\overline t_1)$ is sufficiently small
then $\phi_\mu^{(1)}(y) \ge \phi_\mu^{(2)}(y)$ 
 for all $y \in [\underline t_1, \overline t_1]$. 
 \item If $f_1$ and $f_2$ are uniform densities and 
 $\underline t_1 \le \underline t_2$ and $\overline t_1 \ge \overline t_2$ then $\phi_\mu^{(2)}(y) \ge \phi_\mu^{(1)}(y)$
 for all $y \ge y_1:=(\underline t_1\overline t_2-\underline t_2 \overline t_1)/
 (\underline t_1+\overline t_2-\underline t_2 -\overline t_1)$ (where $y_1 \in [\underline t_2, \overline t_2]$).
 
\end{enumerate}
\end{pro}
\noindent Proposition~\ref{pro:timeintervals}(1) and (3) follow from a more general result
(Proposition~\ref{pro:domination}) which allows to compare $\phi_\mu^{(1)}$ and $\phi_\mu^{(2)}$ 
in more general settings. Proposition~\ref{pro:domination} is fairly technical and can be found
in Section~\ref{sec:proofs}.
Proposition~\ref{pro:timeintervals}(2) implies for instance that a consistent delay of the disturbance
(think of $\underline{t}_2>\overline{t}_1$) reduces the fitness of all individuals following the former ESS.
Even if it is not necessary that $\underline{t}_2>\overline{t}_1$, the requirement that
$\Prob(\underline t_2<D_2<\overline t_1)$ is sufficiently small, implies that $\overline{t}_2>\overline{t}_1$.

%

It is not trivial to assess whether individuals that profit from the climate change represent a small or large fraction of
the population
following what was a former ESS. More precisely, it is not always clear whether the average fitness $\bar\lambda_\mu$
increases or not. Nevertheless it is not difficult to provide some partial answers in the case of delay or anticipation
of the disturbance.
\begin{pro}\label{pro:averagefitnessclimatechange}
 Consider  two densities $f_1$ and $f_2$ on the  intervals 
 $[\underline t_1, \overline t_1]$, $[\underline t_2, \overline t_2]$ and let $\mu$ be the ESS associated to $f_1$.
 Denote by $D_2$ a random variable with density $f_2$.
 \begin{enumerate}
 \item 
 $\bar\lambda_\mu^{(2)}<\bar\lambda_\mu^{(1)}$ provided that
  $\Prob(\underline t_2<D_2<\overline t_1)$ is sufficiently small.
  \item 
 $\bar\lambda_\mu^{(2)}>\bar\lambda_\mu^{(1)}$ provided that
  $F_\mu(\overline t_2)$ is sufficiently small.
 \end{enumerate}
\end{pro}
Proposition~\ref{pro:averagefitnessclimatechange}(1) implies that a consistent delay of the disturbance
reduces the average fitness of the population; on the other hand,
Proposition~\ref{pro:averagefitnessclimatechange}(2) implies that the average fitness of the population
increases if the disturbance arrives so early that most of the population has not yet arrived when the 
danger is over.

Let us discuss now the consequences of a change of $p$.
In the previous sections we showed that $p \mapsto \lambda(a,p)$ is decreasing. This means that if $p$ decreases
and there is an instantaneous adaptation of the population to the new conditions (i.e.~the arrival distribution $\mu$
changes accordingly) then the fitness of each individual (hence the average fitness) increases, 
because more individuals will choose to arrive later in the season.
Intuitively, if the population does not adapt to the climate change, a decrease of $p$ should imply a killing of a larger
fraction of the population and hence a lower average fitness.
%
The following result tells us that this intuition is correct, at least for weak competition.

\begin{pro}\label{pro:changepaveragefitness}
Given a generic arrival strategy $\mu$,
recall the definition of the average fitness as $\bar\lambda_\mu=\bar\lambda_\mu(a,p,\mu,f)=\int \phi_\mu(y) \mu(\dd y)$. The function
$p \mapsto \bar\lambda_\mu=\bar\lambda_\mu(a,p,\mu,f)$ is continuous in $[0,1]$.
Moreover for all $a \in (0, 2 \log(2)]$ 
we have $\partial_p \bar\lambda_\mu(a,p,\mu) \ge 0$ for all $p \in (0,1)$ and the inequality is strict
if and only if $F_\mu(x)>0$ for some $x<\overline t_f$. 
\end{pro}


This means that, if the competition is low, then when the probability of surviving the disturbance decreases
then the fitness decreases as well. Roughly speaking, in case of low competition it is better to
arrive later (and to have more competitors) than to take a chance against the disturbance.

\section{Discussion}\label{sec:discuss}

In this paper we developed a model for the timing of life history events when a disturbance can strike a population 
of migrators and may kill 
some of the individuals that incur in it. Our model therefore 
considered the biologically realistic scenario of a ``soft'' disturbance, i.e.~an event, like extreme unseasonable weather, 
that can kill a fraction of the population, but to which some individuals survive. We developed our study by considering 
migratory birds as a model, because they are biological system where the timing of life history events has been intensively 
studied. In particular, we looked for the ESS for individuals that have to choose arrival time to their breeding grounds, 
and benefit from early arrival, as is often the case for migratory birds (\cite{cf:Crick}).
On the other hand, they may incur in a catastrophe, for instance a spell of cold weather, that will kill a fraction of 
individuals that have arrived to their breeding grounds before the catastrophe. Clearly, the choice of focusing on arrival 
to the breeding grounds is purely exemplificative, and our model applies more generally to the timing of almost any 
life-history event. For example, it applies to timing of arrival to the wintering grounds, crossing of a geographical 
barrier like a mountain range, or to timing of reproduction, rather than arrival.
From a biological point of view, the most interesting results we obtained can be summarized as follows.
First, in presence of competition, a fraction of individuals arrive early during the season, in a period when they 
can incur in a catastrophe and therefore be killed. Interestingly, the stronger the competition, the earlier birds start
arriving and may even arrive at the earliest time possible.
Remarkably, there is a threshold value for the intensity of the competition above which a fraction of individuals arrive 
extremely early, i.e.~at a time when they will certainly incur in the catastrophe. Hence, competition is able to force 
individuals to risk death if the payoff for an early arrival is sufficiently high.
Second, a strong disturbance increases the fitness of individuals (fitness is equal for all individuals that follow the ESS) 
because, under a strong disturbance, the fraction of individuals arriving at a time when they can incur in the catastrophe decreases. Hence, a strong disturbance determines a later average arrival of the population. Actual distribution of arrival dates is therefore the balance between the contrasting pressures of competition and risk of death due to the catastrophe.
Third, the ESS is not the strategy that determines the maximum average fitness of individuals in the population. 
Indeed, a larger average fitness could be obtained if some individuals would accept a reduction in their own fitness for 
favouring other individuals. This result is not surprising since individuals are predicted to behave selfishly and adopt the 
strategy that maximizes their own fitness.
We also tested our model under a climate change scenario, which is predicted to increase the frequency and strength
of extreme meteorological events (\cite{IPCC2013}), like spells of unseasonal weather that can kill migratory birds. 
We obtained two further interesting results.
Fourth, climate change, which is predicted to increase the strength of the disturbance, should determine an increase 
in the fitness when individuals are able to adapt their arrival times to the new ESS (second result above). However, the ESS 
implies a later average arrival of individuals (third result above). Differences among species or populations in the observed 
shifts in the timing of migration according to climate change may therefore be due to differences in the susceptibility of 
species or populations to extreme weather conditions, which in our model is accounted for by the strength of the disturbance. 
Since climate change is determining on the one side a general advancement of the timing of spring events 
(\cite{cf:Menzel, cf:Schwartz}), but on the other side it also increases the frequency of cold spells in 
spring (\cite{cf:Gu}), differences in the response of bird species or populations to climate change should be investigated 
also in respect to their susceptibility to unseasonable weather.
Fifth, if the population is unable to adjust arrival times and continues following the previous timing, which is no more an ESS, 
the fitness of individuals declines (in many scenarios). It has been hypothesized that many migratory bird species, 
particularly long distance migrants, may be less able than short distance ones to advance their arrival to the breeding grounds 
because they are constrained by the timing of other life-history events (\cite{cf:Moller2010}). Consequently, they are forced to 
follow an arrival strategy that differs from the current ESS, and should suffer a reduction in fitness.
Our model therefore gives an explanation of the possible mechanisms linking response to climate change and population trends 
and explaining why bird populations that did not show a response to climate change are declining (\cite{cf:Moller2008}).
There is currently debate among biologists on whether the observed changes in arrival dates of migratory birds can be attributed 
more to micro-evolutionary processes or to phenotypic plasticity (\cite{cf:Charmentier}). If the response to climate change 
is due to phenotypic plastic response of individuals, then probably adaptation will be fast enough to keep the pace of climate 
change. In contrast, if timing of life history events is genetically controlled, a longer time may be needed for the new ESS 
to fix in the population, and in the meanwhile individual fitness will be reduced.
In summary, the model we developed may contribute to our understanding of the processes determining the 
timing of life history events under the biologically realistic scenario of a catastrophe killing only a fraction of 
the individuals that incur in it. Our model explains how competition can induce a fraction of the population to arrive 
very early, despite facing a higher risk of death, as it is documented in several species (\cite{cf:Newton}). 
Moreover, our model also investigated the effect of climate change on the timing of life history events, and 
demonstrated that fitness should decline in a scenario of increased probability of catastrophe if the population is not 
able to adapt to the new climatic conditions.

\section{Proofs}\label{sec:proofs}

In this section one can find all the proofs of our results and some details about the remarks of the previous
sections. We start with a lemma and its corollary. 

\begin{lem}\label{lem:exponential}
Let $\{a_n\}_{n \ge 0}$, $\{b_n\}_{n \ge 0}$ and $\{k_n\}_{n \ge 0}$ be such that $a_n, b_n \ge 0$ for all 
$n \ge 0$, $b_n >0$ for some $n \ge 0$, $k_{n+1} \ge k_n$ for all $n \ge 0$ and $a_n= b_n k_n$ 
for all $n \ge 0$. Define $n_0 :=\min \{n \ge 0 \colon b_n>0\}$ and
\[
 f(t) :=
\begin{cases}
 \displaystyle \frac{\sum_{n \ge 0}a_n t^n}{\sum_{n \ge 0}b_n t^n} & t>0 \\
&\\
\displaystyle \frac{a_{n_0}}{b_{n_0}} & t=0.
\end{cases}
\]
Then $f$ is a nondecreasing function on $[0, +\infty)$. Moreover $f$ is strictly increasing on $[0, +\infty)$ if and only if there exists 
$m,n$ such that $k_{m}>k_n$ and $b_m,b_n >0$.
\end{lem}

\begin{proof}
 Note that $f$ is continuous on $[0,+\infty)$ and differentiable on $(0,+\infty)$.
We compute the derivative on  $(0,+\infty)$ as
\[
 \begin{split}
f^\prime(t) &= \frac{\sum_{n \ge 1}n a_n t^{n-1} \cdot \sum_{n \ge 0}b_n t^n -\sum_{n \ge 0}a_n t^n \cdot \sum_{n \ge 1}n b_n t^{n-1}}{\Big (\sum_{n \ge 0}b_n t^n \Big )^2}\\
&=\frac{\sum_{n \ge 0} \Big [\sum_{i = 0}^n (i+1)(a_{i+1}b_{n-i}-a_{i} b_{i+1}) \Big ] t^n}{\Big (\sum_{n \ge 0}b_n t^n \Big )^2} 
=\frac{\sum_{n \ge 0} \Big [\sum_{i = 0}^n (i+1)b_{i+1}b_{n-i}(k_{i+1}-k_{n-i}) \Big ] t^n}{\Big (\sum_{n \ge 0}b_n t^n \Big )^2}\\
&=\frac{\sum_{n \ge 0} \Big [(k_{n+1}-k_0)(n+1)b_{n+1}b_0+ \sum_{i = 0}^{n-1} (i+1)b_{i+1}b_{n-i}(k_{i+1}-k_{n-i}) \Big ] t^n}{\Big (\sum_{n \ge 0}b_n t^n \Big )^2}\\
 \end{split}
\]
Now $(k_{n+1}-k_0)(n+1)b_{n+1}b_0 \ge 0$ and, by using $j=n-i-1$,
\[
 \begin{split}
  \sum_{i = 0}^{n-1} (i+1)&b_{i+1}b_{n-i}(k_{i+1}-k_{n-i}) \\
&= \frac12 \Big [ \sum_{i = 0}^{n-1} (i+1)b_{i+1}b_{n-i}(k_{i+1}-k_{n-i}) + 
\sum_{j = 0}^{n-1} (n-j)b_{n-j}b_{j-1}(k_{n-j}-k_{j-1}) \Big ]\\
&=\frac12 \Big [ \sum_{i = 0}^{n-1} (i+1)b_{i+1}b_{n-i}(k_{i+1}-k_{n-i}) -\sum_{j = 0}^{n-1} (n-j)b_{j-1}b_{n-j}(k_{j-1}-k_{n-j}) \Big ]\\
&=\frac12 \sum_{i = 0}^{n-1} b_{i+1}b_{n-i}(k_{i+1}-k_{n-i})(i+1-(n-i))= \frac12 \sum_{i = 0}^{n-1} b_{i+1}b_{n-i}(k_{i+1}-k_{n-i})(2i+1-n).\\
 \end{split}
 \]
Since $(k_{i+1}-k_{n-i})(2i+1-n) \ge 0$ for all $n,i$ such that $n>i\ge0$ we have that $f^\prime(t) \ge 0$ for all $t$, hence $f$ is non-decreasing.

Moreover, if $k_{m}>k_n$ and $b_m,b_n >0$ then $m>n$ and $\sum_{i = 0}^{n+m-2} b_{i+1}b_{n-i}(k_{i+1}-k_{n-i})(2i+1-n) \ge b_m b_n (k_m-k_n)(m-n)>0$
(just take $i=m-1$). This implies $f^\prime(t)>0$ for all $t \in (0,+\infty)$, whence $f$ is strictly increasing on $[0,+\infty)$.
On the other hand if $k_n=k_m$ for all $n,m$ such that $b_mb_n>0$, clearly $f(t)=k_0$ for all $t \in [0,+\infty)$.
\end{proof}

\begin{cor}\label{cor:exponential}
 Suppose that $\bar a > a >0$ and $ p \le 1$. The the function $f(t):=(e^{\bar a t}-p)/(e^{at}-p)$ is strictly increasing
in $[0, +\infty)$.
\end{cor}
\begin{proof}
 Apply Lemma~\ref{lem:exponential} using $a_0=b_0=1-p$, $a_n=\bar a^n/n!$ for all $n \ge1$, $b_n=a^n/n!$ for all $n \ge 1$ and $k_n=(\bar a/a)^n$.
\end{proof}

The following is a brief remark which proves some of the properties of the functions $\gamma$ and $\lambda$
(the others are straightforward).

\begin{proof}[Details on Remark~\ref{rem:properties}]
Most of the results about the functions $a_M$, $\gamma$, $\lambda$ and $x_c$ 
follow easily by checking the monotonicity and the the continuity (in each variable separately) of the
l.h.s.~and r.h.s.~of the defining equations.
%
%
%
%
%
%
%
%
%
%
%
%
%
%
We highlight just the main details.
\begin{enumerate}
 \item Since the l.h.s.~of equation~\eqref{eq:aM} is strictly decreasing, continuous 
with respect to $p$ and strictly increasing, continuous with respect to $a_M$ and
since the r.h.s.~is strictly decreasing, continuous with respect to $p$ we have that $p \mapsto a_M(p)$ is strictly decreasing and continuous.
As for the limit we note that, for every fixed $\beta \in (0,+\infty)$,
\[
 p \int_1^{\exp(\beta)} \frac{\dd z}{z^{1-p}-p} \to
\begin{cases}
0 & \textrm{if } p \to 0^+\\
+\infty & \textrm{if } p \to 1^- 
\end{cases}
\]
since $1/(z^{1-p}-p) \downarrow 1/z$ pointwise in the first case and the Bounded Convergence Theorem applies, while 
$1/(z^{1-p}-p) \uparrow +\infty$ pointwise in the second case and the Monotone Convergence Theorem applies. The limits follow
easily by standard arguments.
\item
Since the r.h.s.~of equation~\eqref{eq:xc} is strictly increasing and continuous with respect to $p$ and with respect to $a$ and
since the l.h.s.~is continuous and nonincreasing with respect to $x_c$ we have that the maps $a \mapsto x_c$ and $p \mapsto x_c$
are strictly increasing. 
As for the continuity, note that, by definition,
for every $\beta >x_c(a_0,p_0,f)$ we have
$\int_{\beta}^{\overline t_f}f(x)\dd x <\int_{x_c(a_0,p_0,f)}^{\overline t_f}f(x)\dd x$ hence
\[
\begin{split}
 &\lim_{a \to a_0^-} \Big ( \int_{\beta}^{\overline t_f}f(x)\dd x- \Big( 1-p_0+\Big( \int_1^{\exp(a)}\frac{\dd z}{z^{1-p_0}-p_0} \Big)^{-1}\Big)^{-1} \Big )<0\\
 &\lim_{p \to p_0^-} \Big ( \int_{\beta}^{\overline t_f}f(x)\dd x- \Big( 1-p+\Big( \int_1^{\exp(a_0)}\frac{\dd z}{z^{1-p}-p} \Big)^{-1}\Big)^{-1} \Big )<0.\\
\end{split}
\]
These inequalities imply that, eventually, the maximal solution $x_c(a,p_0,f) \in (x_c(a_0,p_0,f),\beta)$ (resp.~$x_c(a_0,p,f) \in (x_c(a_0,p_0,f),\beta)$).
The limits from the right can be treated analogously by carefully dealing with the intervals 
where
$\beta \mapsto \int_{\beta}^{\overline t_f}f(x)\dd x$ is constant.

Let us prove the limits in equation~\eqref{eq:xclimits}.
By using the monotonicity of $p\mapsto x_c$, we have that $x_c(a,p,f) \le x_c(a,0,f)$. Since the r.h.s.~of equation~\eqref{eq:xc}
equals $a/(1+a)<1$ when $p=0$ we get that $x_c(a,0,f)<\overline t_f$.
In order to compute the second limit 
we observe that the r.h.s.~of equation~\eqref{eq:xc}
tends to $0$ as $a \to 0^+$.
As for the second line in equation~\eqref{eq:xclimits}, note that for every $\beta> \underline t_f $  
we have that $\int_{\beta}^{\overline t_f}f(x)\dd x<1$ hence
 $\int_{\beta}^{\overline t_f}f(x)\dd x-\Big( 1-p+\Big( \int_1^{\exp(a)}\frac{\dd z}{z^{1-p}-p} \Big)^{-1}\Big)^{-1} \to
 \int_{\beta}^{\overline t_f}f(x)\dd x-1<0$ eventually as $p \to p_M(a)^-$. 
\item
The continuity follows from the continuity of the r.h.s.~of equation~\eqref{eq:gamma} with respect to $p$, from the continuity
of the l.h.s.~with respect to $p$, $a$ and $\gamma$ and from the monotonicity with respect to $\gamma$.

Let us prove that $\lim_{a \to \infty} \gamma(a,p)=p$.
%
We define $G(a,p,\gamma):=\exp(-a\gamma)\int_{\exp(a\gamma)}^{\exp(a)}\frac{\dd v}{v^{1-p}-p}$ and
we can assume $p \in (0,1)$.
Clearly, $\int_{\exp(a\gamma)}^{\exp(a)}\frac{\dd v}{v^{1-p}-p} \sim \int_{\exp(a\gamma)}^{\exp(a)}\frac{\dd v}{v^{1-p}}=
(\exp(ap)-\exp(a\gamma p))/p$ as $a \to +\infty$. Hence for every fixed $\gamma$ and $p$ and for every $\varepsilon>0$
we have
\[
 (1+\varepsilon)(\exp(a(p-\gamma))-\exp(a\gamma(p-1)))/p > G(a,p,\gamma)> (\exp(a(p-\gamma))-\exp(a\gamma(p-1)))/p
\]
eventually as $a \to +\infty$.
Hence 
\[
 \lim_{a \to +\infty} G(a,p,\gamma)
 =
 \begin{cases}
+\infty & \textrm{if }\gamma<p \\
0 & \textrm{if }\gamma>p \\
 \end{cases}
\]
thus for every $\bar \gamma$,$ \widetilde \gamma$ satisfying $\bar \gamma<p<\widetilde \gamma$, eventually as $a \to +\infty$,
the solution $\gamma(a,p)$ to equation~\eqref{eq:gamma}, that is $G(a,p,\gamma)-1/p=0$, satisfies $\gamma(a,p)\in (\bar \gamma, \widetilde \gamma)$. This implies
$\lim_{a \to \infty} \gamma(a,p)=p$.

In order to prove that  $\gamma(a,p)<p$ for all $a\ge 0$, $p \in (0,1)$, 
it is enough to show that $G(a,p,p)<1/p$ for all $a\ge 0$, $p \in (0,1)$. Indeed, in that case, the 
 solution $\gamma(a,p)$ of $G(a,p,\gamma)-1/p=0$  satisfies 
 $\gamma(a,p)\in (0,p)$ (note that $\gamma \mapsto G(a,p,\gamma)$ is strictly decreasing for all $a>0$).
Observe that
\[
 \begin{split}
G(a,p,\gamma)&=\frac{\exp(a(p-\gamma))-\exp(-a\gamma(1-p))}{p}-\exp(-a\gamma) 
\int_{\exp(a\gamma)}^{\exp(a)}\Big ( \frac{1}{v^{1-p}-p}-\frac{1}{v^{1-p}} \Big) \dd v \\
 &< \frac{\exp(a(p-\gamma))-\exp(-a\gamma(1-p))}{p}-
 \exp(-a\gamma) 
 \int_{\exp(a\gamma)}^{\exp(a)} \frac{p}{v^{2(1-p)}}\dd v\\
 &= \frac{\exp(a(p-\gamma))-\exp(-a\gamma(1-p))}{p}-\Delta(a,p,\gamma)
 \end{split}
 \]
 where
 \[
 \Delta(a,p,\gamma):=
 \begin{cases}
e^{-a\gamma}a(1-\gamma)/2 & \textrm{if }p=1/2\\
 e^{-a\gamma}\frac{p}{2p-1} \big (
   e^{a(2p-1)}-e^{a \gamma(2p-1)} \big )
   & \textrm{if }p\not =1/2.\\
 \end{cases}
  \]
When $\gamma=p$ we have $G(a,p,p)=1/p- \big ( e^{-ap(1-p)}/p-\Delta(a,p,p) \big )$, hence
it is enough to prove that $e^{-ap(1-p)}/p-\Delta(a,p,p)>0$, that is,
$F_p(a):=pe^{ap(1-p)}\Delta(a,p,p)<1$.
The first case is $\gamma=p=1/2$; thus, $F_p(a)=e^{-a/4}a/8$ which attains its maximum
value (in $[0, +\infty)$) at $a=4$ and $F_p(2)=e^{-1}/2<1$.
The second case is $\gamma=p \not =1/2$ where
$F_p(a)=p^2 \big (e^{-a(1-p)^2}-e^{-ap(1-p)} \big )/(2p-1)$.
Note that $F_p(0)=0$, $F_p(a)>0$ for all $a>0$ and $F_p(a)\to 0$ as $a \to +\infty$. 
Hence $F_p$ admits a global maximum in $[0,+\infty)$, which must be a stationary point
since $F_p$ is differentiable.
By taking the derivative 
$F_p^\prime (a)=p^2 \big (p(1-p)e^{-ap(1-p)}-(1-p)^2e^{-a(1-p)^2} \big )/(2p-1)$ we have that a stationary point
$\bar a$ must satisfy $e^{-\bar ap(1-p)}-(1-p)e^{-\bar a(1-p)^2}/p=0$.
Hence $F_p(\bar a)= p e^{-\bar a(1-p)^2}<1$. This proves that $\gamma(a,p)<p$ for all $p \in (0,1)$
and $a>0$.

Now we prove that $a \mapsto \gamma(a,p)$ is strictly increasing in $[a_M(p),+\infty)$
for all $p \in (0,1)$. 
Since $\gamma \mapsto G(a,p,\gamma)$ is strictly decreasing for every fixed $a>0$, it is enough to prove that
$a \mapsto G(a,p,\gamma)$ is strictly increasing (for fixed $\gamma \le p<1$)
to obtain that $a \mapsto \gamma(a,p)$ is strictly increasing where it is defined by equation~\eqref{eq:gamma}
(by standard arguments for implicitly defined functions).  
By continuity it is enough to prove that $G(\bar a,p,\gamma)>G(a,p,\gamma)$ where $\bar a>a>0$
(the case $a=0$ would follow easily).

To this aim, let us replace the variable in the integral defining $G$ in the following way:
$z=\alpha v+\beta$ where $\alpha=(e^{\bar a}-e^{\bar a \gamma})/(e^{a}-e^{a \gamma})$ and
$\beta = - (e^{\bar a +a\gamma}-e^{a+\bar a \gamma})/(e^{a}-e^{a \gamma})$ (note that $\beta<0$ since $\gamma <1$).
This implies that when $v=e^{a\gamma}$ then $z=e^{\bar a \gamma}$ and when $v=e^{a}$ then $z=e^{\bar a}$.
Whence
\[
 \begin{split}
G(a,p,\gamma)&=e^{-\bar a \gamma} \frac{e^{\gamma(\bar a-a)}}{\alpha} \int_{e^{\bar a \gamma}}^{e^{\bar a}} 
\frac{\dd z}{(\frac{z-\beta}{\alpha})^{1-p}-p}\\
&=  e^{-\bar a \gamma}  \int_{e^{\bar a \gamma}}^{e^{\bar a}} 
\frac{1}{z^{1-p}-p} \Big (\frac{z^{1-p}-p}{(\frac{z-\beta}{\alpha})^{1-p}-p}\cdot \frac{e^{\gamma(\bar a-a)}}{\alpha}
\Big ) \dd z =(*).
\end{split}
\]
Observe that $(e^{\gamma(\bar a-a)})/\alpha=(e^{a(1-\gamma)}-1)/(e^{\bar a(1-\gamma)}-1)$.
Moreover
\[
 \frac{z^{1-p}-p}{(\frac{z-\beta}{\alpha})^{1-p}-p}=\alpha^{1-p} \frac{(\frac{z}{\alpha})^{1-p}-p/\alpha^{1-p}}{(\frac{z+|\beta|}{\alpha})^{1-p}-p}
=\alpha^{1-p} \Big (1-\frac{(\frac{z+|\beta|}{\alpha})^{1-p}-(\frac{z}{\alpha})^{1-p}-p(1-1/\alpha^{1-p})}{(\frac{z+|\beta|}{\alpha})^{1-p}-p}
\Big )
\]
which is strictly increasing with respect to $z$ since $z/\alpha<(z+|\beta|)/\alpha$. Hence $((z+|\beta|)/\alpha)^{1-p}-(z/\alpha)^{1-p}$
is nonincreasing (since $1-p \le 1$) and $(z+|\beta|)/\alpha)^{1-p}-p$ is strictly increasing.
This implies that, for all $z \in [e^{\bar a \gamma}, e^{\bar a})$,
 \[
\begin{split}
  \frac{z^{1-p}-p}{(\frac{z-\beta}{\alpha})^{1-p}-p}\cdot \frac{e^{\gamma(\bar a-a)}}{\alpha} &<
\frac{e^{\bar a (1-p)}-p}{e^{a(1-p)}-p}\cdot \frac{e^{a(1-\gamma)}-1}{e^{\bar a(1-\gamma)}-1}
\le \frac{e^{\bar a (1-p)}-1}{e^{a(1-p)}-1}\cdot \frac{e^{a(1-\gamma)}-1}{e^{\bar a(1-\gamma)}-1} \\
&\le
\frac{e^{\bar a (1-p)}-1}{e^{a(1-p)}-1} \cdot \frac{e^{a(1-p)}-1}{e^{\bar a(1-p)}-1}=1
 \end{split}
\]
where in the second inequality we used $(x-p)/(y-p) \le (x-1)/(y-1)$ for all $x \ge y>1\ge p$, while in the
last inequality we applied Corollary~\ref{cor:exponential} (since $\gamma \le p$). Finally, this yields
\[
 (*) < e^{-\bar a \gamma}  \int_{e^{\bar a \gamma}}^{e^{\bar a}} 
\frac{\dd z}{z^{1-p}-p}=G(\bar a,p,\gamma).
\]
\end{enumerate}
\end{proof}

\begin{rem}\label{rem:lambdadecreasing}
The continuity of the function $\lambda(a,p)$ is easy.
In the interval $[0,a_M(p)]$ the function $a \mapsto \lambda(a,p)$ is strictly decreasing since the integral
in the r.h.s.~of equation~\eqref{eq:deflambda} is strictly increasing for all $p \in (0,1)$. If
$a \in [a_M(p), +\infty)$ then, using equation~\eqref{eq:gamma}, we have that $\lambda(a,p)$ is a solution to
\[
 \lambda \int_{p/\lambda}^{\exp(a)}\frac{\dd z}{z^{1-p}-p}=1
\]
and, since the l.h.s.~of this equation is strictly increasing with respect to $a$ and $\lambda$, standard arguments
imply that $a \mapsto \lambda(a,p)$ is strictly decreasing.

We show now that $p \mapsto \lambda(a,p)$ is strictly decreasing. 
Let us start with the first expression in equation~\eqref{eq:deflambda}. In the following
equation it is easy to show that the derivative with respect to $p$ and the integral with respect to $z$ commute and
\[
\begin{split}
\partial_p \frac{1}{\lambda}&=  -\int_1^{\exp(a)}\frac{\dd z}{z^{1-p}-p} +
(1-p)\int_1^{\exp(a)}\frac{z^{1-p} \ln(z)+1}{(z^{1-p}-p)^2}\dd z\\
&=\int_1^{\exp(a)}\frac{z^{1-p} \ln(z^{1-p})+1-p -(z^{1-p}-p)}{(z^{1-p}-p)^2}\\
&=\int_1^{\exp(a)}\frac{z^{1-p} \ln(z^{1-p})+1 -z^{1-p}}{(z^{1-p}-p)^2} >0
\end{split} 
\]
 for every $a>0$ since the integrand is strictly positive for every $z>1$; indeed the function
 $x \mapsto x \ln(x)+1-x=x(\ln(x)-1)+1$ is differentiable in $(1, +\infty)$ and continuous in $[1,+\infty)$ and
the derivative is $\ln(x)>0$ for all $x>1$. This implies that $p \mapsto \lambda(a,p)$ is strictly decreasing
for every fixed $a>0$.

Let us consider now the second expression for $\lambda$, namely $p\exp(-a\gamma)$, which holds for $p> p_M(a)$ where we recall that
$p_M(a)$ is the unique solution for $\int_1^{\exp(a)}\frac{\dd z}{z^{1-p}-p}=\frac{1}{p}$ with respect to $p$. 
From equation~\eqref{eq:gamma} we have that, for every fixed $a>0$, the function $p \mapsto y(p):= p \exp(-a\gamma(a,p))$
is implicitly defined by the equation $F_a(y,p)=0$ where 
\[
 F_a(y,p)=y \int_{p/y}^{\exp(a)} \frac{\dd v}{v^{1-p}-p}-1.
\]
The solution to the previous equation is uniquely defined in $(p_M,p)$ since $y \mapsto F_a(y,p)$ is strictly increasing for every fixed $p \in (0,1)$ and 
$F_a(p_M ,p)< p_M/p_M -1=0$, $F_a(p,p)=p/p_M-1>0$.
We can compute, using similar arguments as before, 
\[
 \begin{split}
\partial_p F_a(y,p) &= -\frac{1}{(p/y)^{1-p}-p} + y \int_{p/y}^{\exp(a)} \frac{v^{1-p}\ln(v)+1}{(v^{1-p}-p)^2}\dd v\\
&> -\frac{1}{(p/y)^{1-p}-p} + \frac{1}{1-p} y \int_{p/y}^{\exp(a)}\frac{\dd v}{v^{1-p}-p}\dd v\\ 
&=  -\frac{1}{(p/y)^{1-p}-p} + \frac{1}{1-p}>0
 \end{split}
\]
where in the last equality we used the fact that, since $F_a(y,p)=0$, then $y \int_{p/y}^{\exp(a)}\frac{\dd v}{v^{1-p}-p}\dd v=1$.
The last inequality holds since $p > y$. By standard arguments, $\partial_p F_a(y,p)>0$ implies that 
$p \mapsto \lambda(a,p)$ is strictly decreasing for every $a>0$.
\end{rem}

Before proving Theorem~\ref{th:implicit}, we prove that an ESS is of the form $\mu=\gamma\delta_0+(1-\gamma)\nu$, 
where $\nu$ is an absolutely  continuous probability measure.

\begin{lem}\label{lem:implicit}
 Let $\mu$ be an ESS and fix $p \in [0,1)$. Then for some $\gamma\in[0,1]$, $\mu=\gamma\delta_0+(1-\gamma)\nu$, 
where $\nu$ is an absolutely continuous probability measure. 
Moreover $[\min\mathrm{supp}(\mu),\overline t_f] \supseteq \mathrm{supp}(\mu)\supseteq[\min\mathrm{supp}(\mu),\overline t_f]\cap\mathrm{Esupp}(f)$ and if 
$\min\mathrm{supp}(\mu)<\min\mathrm{supp}(\nu)$ then $\mathrm{supp}(\nu) \supseteq \mathrm{Esupp}(f)$.
Finally, $\phi_\mu$ is constant on $[\min\mathrm{supp}(\mu),+\infty)$ and equals $\lambda$.
\end{lem}
\begin{proof}
We already noted that we can write equivalently 
\[
 \phi_\mu(x)=\exp(-aF_\mu(x))\Big[\int_0^x\exp(a(1-p)F_\mu(z))f(z)\dd z+p\int_x^{+\infty}f(z)\dd z
 \Big]
\]
since $\mathrm{Esupp}(f) \subseteq [0,\overline t_f]$. From this equation we can see easily that $\phi$ is right-continuous on $[0,+\infty)$;
moreover it is left-continuous at $x$ if and only if $\mu(\{x\})=0$. 
If $(\alpha,\beta]$ is such that $\mu((\alpha,\beta])=0$ then  $(\alpha,\beta]\cap\mathrm{Esupp}(f)\neq\emptyset$,
 if and only if $\phi(\beta)>\phi(\alpha)$. Indeed, 
 \begin{equation}\label{eq:jumpphi}
  \phi_\mu(\beta)-\phi_\mu(\alpha)=\exp(-aF_\mu(\alpha))
  \Big(\exp(a(1-p)F_\mu(\alpha))-p\Big)\int_\alpha^{\beta}f(z)\dd z>0.
 \end{equation}
As a consequence, if $\mu((\alpha,\beta))=0$ then 
$(\alpha,\beta)\cap\mathrm{Esupp}(f)\neq\emptyset$ if and only if $\lim_{t \to \beta^-}\phi_\mu(t)>\phi_\mu(\alpha)$.
Moreover if $\{x_n\}_{n \in \N}$ is such that $x_n \in \mathrm{supp}(\mu)$, $x_n \not = x$ for all $n \in \N$ and 
$x_n \uparrow x$ then $x \in \mathrm{supp}(\mu)$ and $\mu(\{x\})=0$. The first assertion, namely $x \in \mathrm{supp}(\mu)$,
comes from the fact that the support of a measure is a closed set.
Suppose, by contradiction, that $x$ is an atom; clearly
$\phi(x_n)=\lambda$ for all $n$, since $\mu$ is an ESS, and
$\lim_{t \to x^-} \phi_\mu(t)>\phi_\mu(x)$ hence $\phi_\mu(x)<\lambda$
 which is a contradiction.
This proves that the only atom, if any, of $\mu$ must be $0$. Thus  $\mu=\gamma\delta_0+(1-\gamma)\nu$,
where $\gamma=\mu(0)$ and $\nu$ is nonatomic.

We prove now that if $\alpha<\beta$, $\alpha\in\mathrm{supp}(\mu)$
and $\beta\in\mathrm{Esupp}(f)$, then $\beta\in\mathrm{supp}(\mu)$; this implies  
$\mathrm{supp}(\mu)\supseteq[\min\mathrm{supp}(\mu),\overline t_f]\cap\mathrm{Esupp}(f)$.
By contradiction if $\beta \not \in \mathrm{supp}(\mu)$ there 
exists $\varepsilon>0$
such that $(\beta -\varepsilon, \beta+\varepsilon) \cap \mathrm{supp}(\mu)=\emptyset$. 
Let $\bar\beta:=\max (\mathrm{supp}(\mu)\cap[0,\beta])$ then $\bar\beta<\beta$. 
Since   $\beta\in\mathrm{Esupp}(f)$, $\int_{\beta-\varepsilon}^{\beta+\varepsilon} f(x) \dd x>0$.
Clearly $(\bar \beta,  \beta+\varepsilon) \cap \mathrm{supp}(\mu)=\emptyset$, as a consequence of equation~\eqref{eq:jumpphi},
$0<\lim_{t \to (\beta+\varepsilon)^-}\phi_\mu(t)-\phi_\mu(\bar \beta)=\lim_{t \to (\beta+\varepsilon)^-}\phi_\mu(t)-\lambda$.
Hence there exists $t \in (\beta -\varepsilon, \beta+\varepsilon)$ such that $\phi_\mu(t)>\lambda$
but this contradicts the definition of ESS. 

Let us prove that $[\min\mathrm{supp}(\mu),\overline t_f] \supseteq \mathrm{supp}(\mu)$. 
Observe that, for every $\beta > \overline t_f$ we have
\[
 \phi_\mu(\beta)-\phi_\mu(\overline t_f 
)= \big (\exp(-aF_\mu(\beta))-\exp(-aF_\mu(\overline t_f) \big ) \int_0^{\overline t_f}
\exp(a(1-p)F_\mu(z))f(z)\dd z \le 0. 
\]
On the other hand, if $\beta \in \mathrm{supp}(\mu)$, there exists $\varepsilon \ge 0$ such that $F_\mu(\beta+\varepsilon)> F_\mu(\overline t_f)$
and $\beta+\varepsilon \in \mathrm{supp}(\mu)$.
This implies $\phi_\mu(\beta+\varepsilon) < \phi_\mu(\overline t_f)$ which contradicts the
definition of an ESS. Thus $[\min\mathrm{supp}(\mu),\overline t_f] 
\supseteq \mathrm{supp}(\mu)$.


If $\min\mathrm{supp}(\mu)<\min\mathrm{supp}(\nu)$ then $\gamma>0$ and $0=\min\mathrm{supp}(\mu)<\min\mathrm{supp}(\nu)$.
On the other hand, $\mathrm{supp}(\nu)\subseteq\mathrm{supp}(\mu)\subseteq\mathrm{supp}(\nu)\cup\{0\}$,
that is, $\mathrm{supp}(\mu)\setminus\{0\}=\mathrm{supp}(\nu)\setminus\{0\}$.
Since when $\min\mathrm{supp}(\mu)=0$ we have  $\mathrm{supp}(\mu)\supseteq\mathrm{Esupp}(f)$, then 
$\mathrm{supp}(\nu)\setminus\{0\}\supseteq\mathrm{Esupp}(f)\setminus\{0\}$ which implies
$\mathrm{supp}(\nu)= \overline{\mathrm{supp}(\nu)\setminus\{0\}}\supseteq
\overline{\mathrm{Esupp}(f)\setminus\{0\}}=\mathrm{Esupp}(f)$.

We are left to prove that $\nu$ is absolutely continuous.
First of all we note that, from equation~\eqref{eq:phiexplicit}, $\phi_\mu(x)=\exp(-aF_\mu(x)) H_\mu(x)$
where
\[
H_\mu(x)=\int_0^x\exp(a(1-p)F_\mu(z))f(z)\dd z+p\int_x^{+\infty}f(z)\dd z
\]
is absolutely continuous and nondecreasing on $[0,\overline t_f]$. Since the only atom of $\mu$, if any, is $0$ and $\phi_\mu$ is right-continuous, 
clearly $\phi_\mu$ is continuous on $[0,\overline t_f]$. It is well-known that any open set in $\mathbb{R}$ can be decomposed into a disjoint,
at most countable
union of open intervals, hence $(0,+\infty)=(\mathrm{supp}(\mu)\setminus\{0\}) \cup \bigcup_{j \in J} I_j$ where 
$\{I_j\}_{j \in J}$ is an at most countable
disjoint union of open intervals in $(0,+\infty)$. By definition, $\phi_\mu(x)=\lambda$ for all $x \in \mathrm{supp}(\mu)$.
Since $\mathrm{supp}(\mu)\supseteq(\min \mathrm{supp}(\mu)\cap \mathrm{Esupp}(f) ,
\overline t_f)$ then $\int_{I_j} f(x) \dd x=0$  for all $I_j \subseteq (\min \mathrm{supp}(\mu),
+\infty)$. As a consequence of equation~\eqref{eq:jumpphi}, $\phi_\mu$ is constant on such intervals $I_j$, but, since the value
of $\phi_\mu$ at the extremal points of $I_j$ is $\lambda$, then $\phi_\mu(x)=\lambda$ for all $x \in I_j$.
This proves that $\phi_\mu(x)=\lambda$ for all $x \ge \min  \mathrm{supp}(\mu)$. Hence
$
 \lambda=\exp(-aF_\mu(x))H_\mu(x)
$ for all $x \ge \min  \mathrm{supp}(\mu)$
which implies
\begin{equation}\label{eq:Fmu}
 F_\mu(x)=
 \begin{cases}
\displaystyle \frac{1}{a}\log \Big (\frac{H_\mu(x)}{\lambda}\Big) &  \text{if } x \ge \min  \mathrm{supp}(\mu) \\
 \\
0 & \text{if } x <\min \mathrm{supp}(\mu). 
 \end{cases}
\end{equation}
By composition, $F_\mu$ is clearly absolutely continuous on $(\min \mathrm{supp}(\mu),\overline t_f)$.
Hence $F_\nu$ is absolutely continuous on $(\min \mathrm{supp}(\mu),\overline t_f)$ (since
$F_\mu(x)-(1-\gamma)F_\nu(x)=\gamma$ for all $x \ge 0$)
and this implies
that $\nu$ is an absolutely continuous measure (since $\mathrm{supp}(\nu) \subseteq [\min \mathrm{supp}(\mu),\overline t_f]$).
\end{proof}

Given a measurable subset 
$I \subseteq\mathbb{R}$, we denote by $L^1(I)$ 
 the set of real, measurable
 functions on $I$ which are integrable with respect to the Lebesgue measure. 
When a different measure $\rho$ needs to be 
 specified, we write $L^1(I,\rho)$ instead.

\begin{rem}\label{rem:abscont}
A well-known characterization which will be used in the sequel is the following \textit{Lebesgue's fundamental theorem
of calculus}
(see for instance \cite[Theorem 7.18]{cf:Rud}).
A function $G$ is absolutely continuous on a compact interval $I$ 
if and only if there exists
a function $g \in L^1(I)$ 
such that
for some ($\Longleftrightarrow$ for all) $\alpha \in I$ and for all $\beta \in I$, $G(\beta)-G(\alpha)=\int_\alpha^\beta g(x)\dd x$;
in that case, for almost every $x \in I$, $G$ is differentiable at $x$ and $G^\prime(x)=g(x)$.

This implies that an absolutely continuous function $G$ on $I$ is constant if and only if $G^\prime(x)=0$
for almost every $x \in I$.

Moreover consider two functions $h:I \to J$, $k:J\to Y$; if $k$ is absolutely continuous and $h$ is monotone and 
absolutely continuous then $k \circ h$ is absolutely continuous. Similarly, if $k$ is a Lipschitz function and 
$h$ is absolutely continuous then $k \circ h$ is absolutely continuous. If $h_1:I \to \mathbb{R}$ is absolutely continuous
then $h+h_1$ is absolutely continuous and the same holds for $h \cdot h_1$ if $I$ is compact.

Finally if $k$ is differentiable everywhere
and $h$ is differentiable almost everywhere clearly
we have $(k \circ h)'(x)=k'(h(x))\cdot h'(x)$ almost everywhere and, if in addition  $k \circ h$ is absolutely continuous,
%
%
%
\begin{equation}\label{eq:rem2.2.1}
 k \circ h(x)-k \circ h(x_0)=\int_{x_0}^x (k \circ h)'(z) \dd z = \int_{x_0}^x k'(h(z))\cdot h'(z) \dd z=(*)
\end{equation}
for all $x_0, x \in I$. If, in addition, $h$ is absolutely continuous 
then 
\begin{equation}\label{eq:rem2.2.2}
 \int_{h(x_0)}^{h(x)} k'(y) \dd y=k \circ h(x)-k \circ h(x_0)=(*).
\end{equation}
A function $k$ is locally Lipschitz on $I$ if and only if for every $x_0 \in I$ there exists $\delta>0$ and $M>0$ such that 
$|x-x_0|<\delta$, $|\bar x-x_0|<\delta$ implies $|h(x)-h(\bar x)|\le M |x-\bar x|$. By elementary analysis, if $h$ is a locally Lipschitz function
on $I$ and $I^\prime \subseteq I$ is compact then there exists $M>0$ such that for all $x,\bar x \in I^\prime$ we have $|h(x)-h(\bar x)|\le M |x-
\bar x|$
(that is, $h$ is globally Lipschitz on every compact subset of $I$). It is easy to show that a locally Lipschit function on $I$ is
absolutely continuous on $I$ (one can prove it on every compact subset and then use the fact that 
$I$ is the union of an increasing family of compact subintervals). Hence the following result holds.

Given an interval $I$, if $k$ is locally Lipschitz on $I$, differentiable everywhere and $h$ is absolutely continuous on $I$ 
then the equalities~\eqref{eq:rem2.2.1} and \eqref{eq:rem2.2.2} hold; moreover $k \circ h$ is constant on $I$
if and only if $(k'\circ h)(x)\cdot h'(x)=0$ almost everywhere in $I$.
\end{rem}

\begin{proof}[Proof of Theorem~\ref{th:implicit}]
By Lemma~\ref{lem:implicit}, we know that an ESS can be written as $\mu=\gamma\delta_0+(1-\gamma)\nu$, 
where $\nu$ is an absolutely continuous measure and $\mathrm{supp}(\mu)\supseteq[{x_\mu},\overline t_f]\cap\mathrm{Esupp}(f)$, where 
${x_\mu}:=\min\mathrm{supp}(\mu)$ (we will show later that $x_\mu=x_c$).
Since $F_\mu(y)-(1-\gamma)F_\nu(y)=\gamma$ for all $y \ge 0$, we have that $F_\mu$ is absolutely continuous on
$[0,+\infty)$.

Since $\phi_\mu$ is absolutely continuous, we can take the derivative of $\phi_\mu$ in \eqref{eq:phiexplicit}, 
for almost every $y>0$,
\begin{equation}\label{eq:differential}
 \phi_\mu^\prime(y)=-a(1-\gamma)g(y)\phi_\mu(y)+
 \exp(-aF_\mu(y))\Big[\exp(a(1-p)F_\mu(y))-p \Big]f(y),
\end{equation}
where $g$ is the derivative of $\nu$.
From now on it will be tacitly understood that the derivatives and equalities involving them,
are defined and hold almost everywhere.
From Lemma~\ref{lem:implicit} we know that for all $y>{x_\mu}$ we have $\phi_\mu(y)=\lambda$ hence 
$\phi_\mu^\prime(y)=0$. Thus
\[
   g(y)=\frac{f(y)}{a(1-\gamma)\lambda}\exp(-aF_\mu(y))\Big[\exp(a(1-p)F_\mu(y))-p \Big].   
\]
Let $z:=\exp(aF_\mu)$. Using the last result of Remark~\ref{rem:abscont} (by taking $I=[0,+\infty)$,
$k(x):=\exp(ax)$ and $h:=F_\mu$), the previous equation is equivalent to
\[
z^\prime(y)=\frac{f(y)}{\lambda} (z(y)^{1-p}-p),
\]
which, since $z(y)^{1-p}>p$ for all $y\ge0$, is in turn equivalent to
\[
\frac{z^\prime(y)}{z(y)^{1-p}-p}=\frac{f(y)}{\lambda},
\]
with the condition $z({x_\mu})=\exp(a\gamma)$.
By Remark~\ref{rem:abscont}, this is equivalent to
\begin{equation}\label{eq:almost1}
 \int_{\exp(a\gamma)}^{\exp(a(1-\gamma)F_\nu(x)+a\gamma)}\frac{\dd z}{z^{1-p}-p}=\frac1\lambda \int_{x_\mu}^xf(y)\dd y,
\end{equation}
which, once we prove that $x_c=x_\mu$, is equivalent to  equation~\eqref{eq:implicit}.
This proves that given $\gamma$, $F_\nu$  is uniquely defined by the previous equation.
From the previous equation, for all $x \in [x_\mu,\overline t_f]$, 
\[
 \int_{\exp({a(1-\gamma)F_\nu(x-\varepsilon)+a \gamma})}^{\exp({a(1-\gamma)F_\nu(x+\eps)+a \gamma})}\frac{\dd z}{z^{1-p}-p}=\frac{1}{\lambda}
\int_{x-\eps}^{x+\eps}f(y)\dd y,
\]
which implies
that $\mathrm{supp}(\nu)=[{x_\mu},\overline t_f]\cap\mathrm{Esupp}(f)$.
Now, since $\mathrm{supp}(\mu)\setminus\{0\}=\mathrm{supp}(\nu)\setminus\{0\}$, from Lemma~\ref{lem:implicit} we have that either
$\mu=\nu$ (hence $\mathrm{supp}(\mu)=[{x_\mu},\overline t_f]\cap\mathrm{Esupp}(f)$ and $\min\mathrm{supp}(\mu)=\min\mathrm{supp}(\nu)$)
or $\gamma>0$ and $x_\mu=0$. In this case  $\mathrm{supp}(\nu)=\mathrm{Esupp}(f)$ and $\mathrm{supp}(\mu)=\mathrm{Esupp}(f)\cup\{0\}$.

In particular, since $x_\mu \in \mathrm{supp}(\mu)$, when $\mu=\nu$ (that is, $\gamma=0$) we have that $x_\mu \in \mathrm{Esupp}(f)$.
Clearly the measure $\nu$ is supported
in $[0,\overline t_f]$, but it is a probability measure if and only if $F_\nu(\overline t_f)=1$, 
therefore, by using equation~\eqref{eq:almost1}, ${x_\mu}$ and $\gamma$ must satisfy
\begin{equation}\label{eq:alambda}
  \int_{\exp(a\gamma)}^{\exp(a)}\frac{\dd z}{z^{1-p}-p}=\frac1\lambda \int_{x_\mu}^{\overline t_f}f(y)\dd y.
\end{equation}

We note that $\lambda=\phi_\mu(\min \mathrm{supp}(\mu))=\phi_\mu(x_\mu)$. 
If $\gamma=0$ then $\mu=\nu$ hence, from equation~\eqref{eq:Fmu}, by continuity of $F_\mu$ at
$x=\min \mathrm{supp}(\mu)\equiv x_\mu$ we have $\lambda=H_\mu(x_\mu)$, that is
\begin{equation}\label{eq:xcandlambda}
\lambda=\int_0^{x_\mu}f(y)\dd y+p\int_{x_\mu}^{\overline t_f}f(y)\dd y =
1-(1-p)\int_{x_\mu}^{\overline t_f}f(y)\dd y
\end{equation}
since $\exp(a(1-p)F_\mu(y))=1$ for all $y \in [0,x_\mu)$.

If $\gamma>0$ then $x_\mu=0$ and, from equation~\eqref{eq:phiexplicit},  
$\lambda=\phi_\mu(0)=p\exp(-a\gamma)$ which is the second line of equation~\eqref{eq:deflambda}.

Let $a\le a_M$. If $\gamma>0$ then by  equation~\eqref{eq:alambda} and the discussion thereafter,
\[
 \begin{split}
  \frac1p&= \int_{1}^{\exp(a_M)}\frac{\dd z}{z^{1-p}-p}\ge\int_{1}^{\exp(a)}\frac{\dd z}{z^{1-p}-p}\\
& >\int_{\exp(a\gamma)}^{\exp(a)}\frac{\dd z}{z^{1-p}-p}=\frac1\lambda>\frac1p,
 \end{split}
\]
which is a contradiction.
If $\gamma=0$ then by plugging the explicit value of $\lambda=1-(1-p)\int_{x_\mu}^{\overline t_f}f(y)\dd y$ into
\eqref{eq:alambda}, we get \eqref{eq:xc} (with $x_\mu$ instead of $x_c$) and since 
we showed that $x_\mu \in \mathrm{Esupp}(f)$, 
thus it must be the (unique) maximal solution to equation~\eqref{eq:xc}, that is
$x_\mu=x_c$. This proves that $a\le a_M$ implies $\gamma=0$ and $x_c \ge \underline t_f$. 
Since the r.h.s.~of equation~\eqref{eq:xc} is strictly less than 1 (resp.~equal to 1) when $a<a_M$ (resp.~$a=a_M$)
then equation~\eqref{eq:xc} yields $x_c > \underline t_f $ (resp.~$x_c =\underline t_f $).
In order to obtain the first line of equation~\eqref{eq:deflambda} just 
consider the expression of $\int_{x_c}^{\overline t_f}f(y)\dd y$ given by the r.h.s.~of equation~\eqref{eq:xc} 
and plug it in equation~\eqref{eq:xcandlambda} (recalling that $x_\mu=x_c$).

Let $a> a_M$. If $\gamma=0$ then, by the definition of $a_M$, equation~\eqref{eq:alambda} and the discussion thereafter,
\[
 \begin{split}
  \frac1\lambda \int_{x_\mu}^{\overline t_f}f(y)\dd y & = \frac{\int_{x_\mu}^{\overline t_f}f(y)\dd y}{1-(1-p)\int_{x_\mu}^{\overline t_f}f(y)\dd y} \le\frac1p
=\int_{1}^{\exp(a_M)}\frac{\dd z}{z^{1-p}-p}\\
&<\int_{1}^{\exp(a)}\frac{\dd z}{z^{1-p}-p}=\frac1\lambda \int_{x_\mu}^{\overline t_f}f(y)\dd y,
 \end{split}
\]
which is a contradiction.
If $\gamma>0$ then $x_\mu=0$ (which coincides with the definition of $x_c$ given before equation~\eqref{eq:xc}) and
by plugging the explicit value of $\lambda=p\exp(-a\gamma)$ in equation~\eqref{eq:alambda} we have
\[
 \frac1p=\exp(-a\gamma)\int_{\exp(a\gamma)}^{\exp(a)}\frac{\dd z}{z^{1-p}-p}=\int_{1}^{\exp(a(1-\gamma))}
\frac{\dd z}{(z\exp(a\gamma))^{1-p}-p},
\]
which has a unique solution $\gamma$. Note that the r.h.s.~of the equation is a decreasing continuous function of $\gamma$,
say $K(\gamma)$;
moreover $K(0)>1/p$ by \eqref{eq:aM} and 
\[
 K(1-a_M/a)=\int_1^{\exp(a_M)}\frac{\dd z}{(z\exp(a-a_M))^{1-p}-p}<\int_1^{\exp(a_M)}\frac{\dd z}{z^{1-p}-p}=\frac1p.
\]
This implies that the unique solution $\gamma\in(0,1-a_M/a)$; $\gamma < p$ was proved in Remark~\ref{rem:lambdadecreasing}.
\end{proof}

Before giving the details on Remark~\ref{rem:maximal} we need a Lemma and another remark.

\begin{lem}\label{lem:stochdomin}
 Let $\mu$ and $\nu$ be two finite measures on $\mathbb{R}$.
Then the following conditions are equivalent:
\begin{enumerate}[(1)]
 \item 
$\int h(x) \mu(\dd x) \ge \int h(x) \nu(\dd x)$ for every nondecreasing, measurable $h \in L^1(\mathbb{R},\mu)\cap L^1(\mathbb{R},\nu)$;
\item
$\mu((x,+\infty)) \ge \nu ((x,+\infty))$ for
every $x \in \mathbb{R}$ and $\mu(\R)=\nu(\R)$;
\item
$\int h(x) \mu(\dd x) \le \int h(x) \nu(\dd x)$ for every nonincreasing, measurable $h \in L^1(\mathbb{R},\mu)\cap L^1(\mathbb{R},\nu)$;
\item
$\mu((-\infty,x]) \le \nu ((-\infty,x])$ for
every $x \in \mathbb{R}$  and $\mu(\R)=\nu(\R)$.
\end{enumerate}
If one of these conditions holds
we write  $\mu \succeq \nu$.

Moreover the inequality in (1) is strict if and only if there exists $y$ such that 
$\mu\big (h^{-1}([y,+\infty))\big )<\nu\big (h^{-1}([y,+\infty))\big )$ (that is, if and only if there exists $y$ such that 
$\mu\big (h^{-1}((y,+\infty))\big )<\nu\big (h^{-1}((y,+\infty))\big )$).
 
Finally the inequality in (3) is strict if and only if there exists $y$ such that 
$\mu\big (h^{-1}([y,+\infty))\big )<\nu\big (h^{-1}([y,+\infty))\big )$ (that is, if and only if there exists $y$ such that 
$\mu\big (h^{-1}((y,+\infty))\big )<\nu\big (h^{-1}((y,+\infty))\big )$).
\end{lem}

\begin{proof}
 The equivalence between  (1) and (2) (or between (3) and (4)) is a classical  result of measure theory:
 it is a slight modification of the arguments in \cite[Section 1.A.1]{cf:ShSh07}. Moreover, to prove that
$\mu(\mathbb{R})=\nu(\mathbb{R})$ just take condition (1) (or condition (3)) and consider first $h\equiv 1$
and then $h \equiv -1$. Under $\mu(\mathbb{R})=\nu(\mathbb{R})$ the equivalence between conditions (2) and (4)
is trivial.
In particular the equivalence between the previous four conditions hold even if we take 
the set of nonnegative (or nonpositive) measurable functions instead of $L^1(\mathbb{R},\mu)\cap L^1(\mathbb{R},\nu)$.

Finally note that 
$y \mapsto \mu([y,+\infty))$ and $y \mapsto \nu ([y,+\infty))$ are continuous from the left while
 $y \mapsto \mu((y,+\infty))$ and $y \mapsto \nu ((y,+\infty))$ are continuous from the right;
this implies that  there exists $y$ such that 
$\mu\big (h^{-1}([y,+\infty))\big )<\nu\big (h^{-1}([y,+\infty))\big )$ if and only if there exists $y^\prime$ such that 
$\mu\big (h^{-1}((y^\prime,+\infty))\big )<\nu\big (h^{-1}((y^\prime,+\infty))\big )$.
%
%
%
%
%
%
\end{proof}

\begin{rem}\label{rem:quasiuniform}
Given a random variable $X$ with law $\mu$, then the composition $\xi:=F_\mu \circ X$ between the variable and its cumulative distribution function
is a random variable with values in $\mathrm{Rg}(F_\mu) \subseteq [0,1]$
such that
$\pr(\xi \le t)\le t$ for all $t \in [0,1]$; moreover $\pr(\xi \le t)= t$ if and only if $t \in \mathrm{Rg}(F_\mu):=F_\mu(\mathbb{R})$.
More generally $\pr(\xi \le t)= \sup\{F_\mu(r) \colon F_\mu(r) \le t\}$, thus if 
$\mu(z)>0$ then $\pr(\xi \le t)=\lim_{s \to z^-} F_\mu(s)=F_\mu(z)-\mu(z)$ for all $t \in [\lim_{s \to z^-} F_\mu(s), F_\mu(z))\equiv [F_\mu(z)-\mu(z), F_\mu(z))$.

To be precise if we 
consider $\{x \colon \mu(x)>0\}=:\{x_i\colon i \in J\}$, the at-most-countable set of discontinuity points of $F_\mu$
(where $J \subseteq \mathbb{N}$),
and define $I:=\bigcup_{i \in J} [F_\mu(x_i)-\mu(x_i),F_\mu(x_i))$ we have
\[
 F_\xi(t)=
\begin{cases}
 t & \textrm{if } t \not \in I\\
F_\mu(x_i)-\mu(x_i) & \textrm{if } t \in [F_\mu(x_i)-\mu(x_i),F_\mu(x_i)).
\end{cases}
\]
This implies in particular that if $F_\mu$ is continuous then the random variable $\xi$ is uniformly distributed on $[0,1]$ otherwise,
according to Lemma~\ref{lem:stochdomin},
$\xi \succ \mathrm{Unif}(0,1)$ (that is, the law of $\xi$ stochastically dominates a the uniform distribution on $[0,1]$).
\end{rem}

\begin{proof}[Details on Remark~\ref{rem:maximal}]
For a generic $f$, 
by performing a time rescaling, we can assume without loss of generality that $\overline t_f=1$. Consider the following family of
absolutely continuous (with respect to the Lebesgue measure) measures $\mu_n(\cdot)=n |\cdot \cap[1-1/n,1]|$ with support in $[0,1]$.
(where $|\cdot|$ represents the Lebesgue measure on $\mathbb{R})$. If we define 
$c_n(f,p):= \int_0^{1-1/n} f(z) \dd z +p \int_{1-1/n}^1 f(z) \dd z$ then $c_n(f,p) \uparrow 1$ as $n \to \infty$ and we have
 \[
 \begin{split}
\int \phi_{\mu_n}(y) \mu_n(\dd y)&= \int_{1-1/n}^1 e^{-an(y-1+1/n)} \Big [c_n(f,p)+ \int_{1-1/n}^y (e^{a(1-p)n(z-1+1/n)}-p)f(z) \dd z
\Big ] n \dd y\\
& \ge \int_{1-1/n}^1 e^{-an(y-1+1/n)} c_n(f,p) n \dd y = c_n(f,p)\int_0^1 e^{-ay^\prime} \dd y^\prime \uparrow \frac{1-\exp(-a)}{a}
 \end{split}
\]
as $n \to \infty$ (according to the Monotone 
Convergence Theorem). This example can be modified in many ways: for instance, one can take a family of absolutely continuous measures,
with strictly positive densities on $[0,1]$, which approximate the previous densities  $n\ident_{[1-1/n,1]}$ in the $L^1$ norm 
(we do not give more details on this).

We are left to prove that $\int \phi_\mu(y) \mu(\dd y)<(1-\exp(-a))/{a}$ for all $\mu$.
We start by proving that for every measurable, nonincreasing function $h$ and every $z \in \mathbb{R}$ we have
$\int_{(z, +\infty)} h(F_\mu(y)) \mu(\dd y) \le \int_{(F_\mu(z),1]} h(y) \dd y$ (and, if $h$ is 
strictly decreasing on $[0,1]$,
then the inequality is
strict if and only if $F_\mu$ is not continuous on $(z,+\infty)$). 
By stochastic domination (see Remark~\ref{rem:quasiuniform}), if we denote by $\pr_\xi$ the law of $\xi$, 
$\int_{\mathbb{R}} h(F_\mu(y)) \mu(\dd y)=\int_{[0,1]}h(s) \pr_{\xi}(\dd s) \le \int_{[0,1]} h(s) \dd s$ 
and, for a 
strictly decreasing function, the inequality
is strict if and only if $F_\mu$ is not continuous on $\mathbb{R}$
(apply Lemma~\ref{lem:stochdomin} to $\pr_\xi$ and to the uniform distribution on $[0,1]$ 
and choose $y:=h(s)$ for some $s \in [\lim_{w \to z^-} F_\mu(w), F_\mu(z))$ whenever this interval is nonempty;
note that, for a 
strictly decreasing function, $h^{-1}([h(s), +\infty))=(-\infty,s]$
for all $s$).
Take $z  \in \mathbb{R}$; then
\begin{equation}\label{eq:equalitypxi}
\int_{(F_\mu(z),1]} h(s) \pr_{\xi}(\dd s)=\int_{F_\mu^{-1}((F_\mu(z),1])} h(F_\mu(y)) \mu(\dd y) =\int_{(z, +\infty)} h(F_\mu(y)) \mu(\dd y)
\end{equation}
since $F_\mu^{-1}((F_\mu(z),1]) \supseteq (z,+\infty)$ and $\mu \big (F_\mu^{-1}((F_\mu(z),1])\setminus (z,+\infty) \big )=0$.


By the previous equation we just need to prove that $\int_{(F_\mu(z),1]} h(s) \pr_{\xi}(\dd s)
\le \int_{(F_\mu(z),1]} h(s) \dd s$ and this can be done in two different ways.

The quick way is to note that, according to Remark~\ref{rem:quasiuniform}, $\pr_\xi([0,F_\mu(z)])$ equals $F_\mu(z)$, 
that is, the lebesgue measure of the interval $[0,,F_\mu(z)]$;
hence the measure $\pr_\xi$ restricted to $(F_\mu(z), 1]$ dominates the Lebesgue measure on the same interval, thus
Lemma~\ref{lem:stochdomin} applies.

An alternate proof is as follows; if we define
\[
 \bar h(y) :=
\begin{cases}
 h(s) & \text{if } s \ge F_\mu(z) \\
h(F_\mu(z)) &\text{if } s < F_\mu(z)\\ 
\end{cases}
\]
then $\bar h$ is nonincreasing hence, by stochastic domination, $\int_{\mathbb{R}} \bar h(F_\mu(y)) \mu(\dd y) \le \int_{[0,1]} \bar h(s) \dd s$.
Thus
\begin{equation}\label{eq:domination}
 \begin{split}
\int_{(z, +\infty)} h(F_\mu(y)) \mu(\dd y)&=\int_{(z, +\infty)} \bar h(F_\mu(y)) \mu(\dd y)=
\int_{\mathbb{R}} \bar h(F_\mu(y)) \mu(\dd y)-\int_{(-\infty,z]} \bar h(F_\mu(y)) \mu(\dd y)\\
&=  \int_{\mathbb{R}} \bar h(F_\mu(y)) \mu(\dd y)-\bar h(F_\mu(z)) \mu((-\infty,z])\\
& \le 
\int_{[0,1]} \bar h(s) \dd s-\bar h(F_\mu(z)) \mu((-\infty,z])=\int_{[0,1]} \bar h(s) \dd s-\bar h(F_\mu(z)) F_\mu(z)\\
&=\int_{[0,1]} \bar h(s) \dd s-\int_{[0,F_\mu(z)]} \bar h(s) \dd s\\
&=\int_{(F_\mu(z),1]} \bar h(y) \dd y=\int_{(F_\mu(z),1]} h(y) \dd y\\
 \end{split}
\end{equation}
and, if $h$ is strictly decreasing on $[0,1]$, the inequality is strict if and only if $F_\mu$ is not continuous on $(z,+\infty)$
(since $\bar h$ is strictly decreasing on $[F_\mu(z),1]$ we use equation~\eqref{eq:equalitypxi} 
and we can apply again Lemma~\ref{lem:stochdomin}  
to $\pr_\xi$ and to the uniform distribution on $[0,1]$ by choosing 
$y:=h(s)$ for some $s \in [\lim_{w \to \bar z^-} F_\mu(w), F_\mu(\bar z))$ whenever this interval is nonempty and $\bar z \in (z +\infty)$).
%
%



Clearly
\[
\begin{split}
\int \phi_\mu(y) \mu(\dd y)= \mathbb{E} \Big [\exp(-a\xi)\Big[\int_0^X\exp(a(1-p)F_\mu(z))f(z)\dd z+p\int_X^{\overline t_f}f(z)\dd z
 \Big]
\Big ].\\
\end{split}
\]
Hence, using Fubini's Theorem and the inequality~\eqref{eq:domination} with $h(s):=e^{-as}$, we have
\[
\begin{split}
\int \phi_\mu(y) \mu(\dd y)
&= \int
e^{-aF_\mu(y)} \Big [ \int_0^y  (e^{a(1-p)F_\mu(z)}-p) f(z) \dd z +p
\Big ]
\mu(\dd y)\\
& = p\int
e^{-aF_\mu(y)}\mu(\dd y) + \int  (e^{a(1-p)F_\mu(z)}-p) f(z) \int_{(z, +\infty)} 
e^{-aF_\mu(y)} 
\mu(\dd y) \dd z\\
&\le  p\int_{[0,1]} e^{-a y} \dd y+
\int  (e^{a(1-p)F_\mu(z)}-p) f(z) \int_{(F_\mu(z),1]} e^{-a y} \dd y \dd z\\
&= p \frac{1-e^{-a}}{a}+
\int  (e^{a(1-p)F_\mu(z)}-p) f(z) \frac{e^{-aF_\mu(z)} -e^{-a}}{a}\dd z\\
&= p \frac{1-e^{-a}}{a}+
\int  (e^{-apF_\mu(z)}-pe^{-aF_\mu(z)}) f(z) \frac{1 -e^{-a(1-F_\mu(z))}}{a}\dd z
=(**)
\end{split}
\]
and the inequality is strict if and only if $F_\mu$ is not continuous.
Now $(1 -e^{-a(1-F_\mu(z))})/a \le (1 -e^{-a})/a$ and 
$e^{-apF_\mu(z)}-pe^{-aF_\mu(z)} \le 1-p$ (since $s \mapsto e^{-aps}-pe^{-as}$ is decreasing in $[0,+\infty)$); 
moreover each of the previous inequalities become a strict inequality
if and only if $F_\mu(z)>0$. 
Whence
\[
\begin{split}
(**)
& \le 
p \frac{1-e^{-a}}{a} +(1-p) \frac{1-e^{-a}}{a}
\end{split}
\]
and there is a strict inequality if and only if $\inf \mathrm{supp}(\mu) < \sup \mathrm{Esupp}(f)$.
\end{proof}

\begin{proof}[Proof of Proposition~\ref{pro:timeintervals}]
 \begin{enumerate}
  \item  Applying Proposition~\ref{pro:domination},
 we have easily that $\int_z^y f_1(x)\dd x \ge \int_z^y f_2(x) \dd x$ for all 
$z \le y$, $y \in [\underline t_1, \underline t_2]$, hence $\phi_\mu^{(1)}(y) \ge \phi_\mu^{(2)}(y)$.

\item
If $\underline t_2 \ge \overline t_1$ the conclusion follows from Proposition~\ref{pro:timeintervals}(1). Let us suppose 
$\underline t_2 < \overline t_1$.
Now let us evaluate $\phi_\mu^{(2)}(y)$ for $y\in[\underline t_2,\overline t_1]$:
\[\begin{split}
 \phi_\mu^{(2)}(y)&=\exp(-aF_\mu(y))\int_{\underline t_2}^y\exp(a(1-p)F_\mu(x))f_2(x)dx+
 p\exp(-aF_\mu(y))\Prob(D_2\in[y,\overline t_1))\\
 &+p\exp({-a})\Prob(D_2\ge\overline t_1)=(\ast).
 \end{split}
\] 
Note that this equality holds even if $\overline{t}_1\ge \overline{t}_2$ since
in that case $\Prob(D_2\ge\overline t_1)=0$.
\[\begin{split} 
 (\ast)&\le \exp(-apF_\mu(y))\Prob(D_2\in[\underline t_2,y])+p\exp(-aF_\mu(y))\Prob(D_2\in[y,\overline t_1])+
 p\exp(-a)\Prob(D_2\ge\overline t_1)\\
 &\le \Prob(D_2\in[\underline t_2,\overline t_1])+p\exp(-a)\Prob(D_2\ge\overline t_1)
 \end{split}
\]
which is smaller than $\exp(-a)$ if
$\Prob(D_2\in[\underline t_2,\overline t_1])$ is sufficiently small. 
When $\mu$ is an ESS, then for all $y\in[\underline{t}_1,\overline{t}_1]$ we have
$\phi_\mu^{(1)}(y)=\lambda(a,p)\ge\exp(-a)$.
If $\underline{t}_2\le\underline{t}_1$ then the proof is complete.
If $\underline{t}_2>\underline{t}_1$ then the conclusion follows using Proposition~\ref{pro:timeintervals}(1)
for $y\in [\underline{t}_1,\underline{t}_2]$.
  \item 
 Consider the solution $y_1$ to the equation 
   $(y-\underline t_1)/(\overline t_1-\underline t_1)=(y-\underline t_2)/(\overline t_2-\underline t_2)$, namely
   $y_1:= (\underline t_2\overline t_1-\underline t_1 \overline t_2)/
 (\underline t_2+\overline t_1-\underline t_1 -\overline t_2)$; clearly for all $y \ge y_1$
 we have
 $(y-\underline t_1)/(\overline t_1-\underline t_1)\le (y-\underline t_2)/(\overline t_2-\underline t_2)$, thus
 Proposition~\ref{pro:domination} implies $\phi_\mu^{(2)}(y) \ge \phi_\mu^{(1)}(y)$.
 \end{enumerate}
%
%
\end{proof}

\begin{pro}\label{pro:domination}
 Consider two probability densities $f_1$ and $f_2$. Let us define $\phi_\mu^{(1)}$ and $\phi_\mu^{(2)}$ according
 to equation~\eqref{eq:phiexplicit} using $f_1$ and $f_2$ respectively.
 Fix $y >0$ such that $\int_0^y f_1(x) \dd x>0$ or $\int_0^y f_2(x) \dd x>0$.
 \begin{enumerate}
 \item Define the set of strong maxima from the left of $F_\mu$ as 
 $M_\mu(y):=\{z \in [0,y] \colon F_\mu(s)<F_\mu(z) \textrm{ for all }s<z \}$.
 If
 \begin{equation}\label{eq:int1}
  \int_z^y f_1(x) \dd x \ge \int_z^y f_2(x) \dd x, \quad \forall z \in M_\mu(y)
 \end{equation}
 then $\phi_\mu^{(1)}(y) \ge \phi_\mu^{(2)}(y)$. In this case  $\phi_\mu^{(1)}(y) = \phi_\mu^{(2)}(y)$ if and only if
 equality holds in equation~\eqref{eq:int1} for all 
 $z \in M_\mu(y)$.
 \item
 If $\int_0^y f_1(x) \dd x>\int_0^y f_2(x) \dd x$ and $a$ is sufficiently small then $\phi_\mu^{(1)}(y) > \phi_\mu^{(2)}(y)$.
 \end{enumerate}
\end{pro}

\begin{proof}[Proof of Proposition~\ref{pro:domination}]
According to equation~\eqref{eq:phiexplicit} we have
 \[
  \phi_\mu^{(i)}(y)=\exp(-aF_\mu(y))\Big[\int_0^y (\exp(a(1-p)F_\mu(x))-p)f_i(x)\dd x+p\Big], \quad i=1,2.
 \]
\begin{enumerate}
 \item  We just need to prove that $\int_0^y (\exp(a(1-p)F_\mu(x))-p)f_1(x)\dd x\ge \int_0^y (\exp(a(1-p)F_\mu(x))-p)f_2(x)\dd x$.
 This follows easily from the fact that $x \mapsto \exp(a(1-p)F_\mu(x))-p$ is a nondecreasing function and applying
Fubini's Theorem.
Indeed, 
for every $z \ge 0$, the set $\{x \ge 0 \colon \exp(a(1-p)F_\mu(x))-p \ge z\}$
is an interval of type $[x_0(z), y]$ (for some $x_0(z) \in [0,y]$); clearly $x_0(z)=0$ for all $z \in [0,1-p]$.
Hence 
 \begin{equation}\label{eq:1}
\begin{split}
\int_0^y (\exp(a(1-p)&F_\mu(x))-p)f_i(x)\dd x
=\int_0^y \Big [\int_0^{\exp(a(1-p)F_\mu(x))-p} \dd z \Big ]f_i(x)\dd x\\
&=\int_0^{\exp(a(1-p)F_\mu(y))-p} \Big [ \int_{\{x \in \mathbb{R}  \colon \exp(a(1-p)F_\mu(x))-p \ge z\}} f_i(x) \dd x\Big ] \dd z\\
&= \int_0^{\exp(a(1-p)F_\mu(y))-p} \Big [ \int_{[x_0(z),y]}  f_i(x) \dd x \Big ] \dd z\\
&=(1-p)\int_{[0,y]}  f_i(x) \dd x +
\int_{1-p}^{\exp(a(1-p)F_\mu(y))-p} \Big [ \int_{[x_0(z),y]}  f_i(x) \dd x \Big ] \dd z.
 \end{split}
\end{equation}
 The equivalence of the equalities is as follows. The ``if'' part is trivial. As for the reverse implication,
 note that $z \mapsto x_0(z)$ is left-continuous and that 
 $s=x_0(z)$ for some $z \in [0, \exp(a(1-p)F_\mu(y))-p]$ if and only if $s=\inf\{t \ge 0 \colon F_\mu(t) \ge \alpha\}$ for
 some $\alpha \le F_\mu(y)$, that is, if and only if $x_0(z) \in M_\mu(y)$. If 
 $\int_{0}^y f_1(x) \dd x > \int_{0}^y f_2(x) \dd x$ then $\phi_\mu^{(1)}(y)> \phi_\mu^{(2)}(y)$ 
follows from equation~\eqref{eq:1}. If
 $\int_{x_0(z)}^y f_1(x) \dd x > \int_{x_0(z)}^y f_2(x) \dd x$ for some $x_0(z)>0$ then by the continuity of 
 $w \mapsto \int_{w}^y f_1(x) \dd x$ and the left continuity of $z \mapsto x_0(z)$
 there exists
 $\varepsilon>0$ such that $\int_{x_0(s)}^y f_1(x) \dd x > \int_{x_0(s)}^y f_2(x) \dd x$ for all $s \in (z-\varepsilon, z]$;
 again equation~\eqref{eq:1} yields the strict inequality $\phi_\mu^{(1)}(y) > \phi_\mu^{(2)}(y)$.
 \item
 It follows from equation~\eqref{eq:1} 
 \[
 \begin{split}
\phi_{\mu}^{(1)}(y)-\phi_{\mu}^{(2)}(y) &\ge (1-p) \int_{[0,y]}  (f_1(x)-f_2(x)) \dd x -(\exp(a(1-p)F_\mu(y))-1).\\
\end{split}
 \]

 by taking the limit as $a$ goes to $0$. 
 \end{enumerate}
\end{proof}

\begin{proof}[Proof of Proposition~\ref{pro:averagefitnessclimatechange}]
 \begin{enumerate}
  \item It follows from the first part of Proposition~\ref{pro:timeintervals}.
  \item
  If $\underline{t}_1>\overline{t}_2$ then, by Remark~\ref{rem:maximal}, $\bar\lambda_\mu^{(2)}= (1-\exp(-a))/a$, which
  is larger than  $\bar\lambda_\mu^{(1)}$ and the proof is complete. 
  If $\underline{t}_1\le\overline{t}_2$, write $\bar\lambda_\mu^{(1)}=(1-\exp(-a))/a-\varepsilon$. 
  Note that, if $y\ge\overline t_2$, then
  $\phi_\mu^{(2)}(y)=\exp(-aF_\mu(y))\int_{\underline t_2}^{\overline t_2}\exp(a(1-p)F_\mu(x))f_2(x)dx$.
Thus, if we integrate with respect to $\mu$:
 \[\begin{split}
    \bar\lambda_\mu^{(2)}&\ge\int_{\underline{t}_2}^{\overline{t}_1}\phi_\mu^{(2)}(y)d\mu(y)\\
    & = \frac1a(\exp(-aF(\overline t_2)-exp(-a))\Big(\int_{\underline t_2}^{\underline t_1}f_2(x)dx
    +\int_{\underline t_1}^{\overline t_2}\exp(a(1-p)F_\mu(x))f_2(x)dx\Big)\\
    &\ge\frac1a(\exp(-aF_\mu(\overline t_2))-\exp(-a))\ge \frac1a(1-\exp(-a))-\varepsilon
   \end{split}
 \]
 where the last inequality holds if $F_\mu(\overline t_2)$ is sufficiently small.
  \end{enumerate}
\end{proof}

\begin{proof}[Proof of Proposition~\ref{pro:changepaveragefitness}]
By using some basic results of mathematical analysis, it is easy to check that
$
  \partial_p \int \phi_\mu(y) \mu(\dd  y) =   \int \partial_p \phi_\mu(y) \mu(\dd  y)
$
 where, from equation~\eqref{eq:phiexplicit},
 \[
  \partial_p \phi_\mu(y)=
  \exp(-aF_\mu(y))\Big[-a \int_0^y\exp(a(1-p)F_\mu(x))F_\mu(x)f(x)\dd x+\int_y^{\overline t_f}f(x)\dd x \Big].
 \]
 From equation~\eqref{eq:domination} we have 
 \begin{equation}\label{eq:domination1}
  \int_{(z,+\infty)} e^{-aF_\mu(y)} \mu(\dd y)\le \int_{(F_\mu(z),1]} e^{-ay} \dd y =\frac{e^{-aF_\mu(z)}-e^{-a}}{a}. 
 \end{equation}
According to Fubini-Tonelli's Theorem  
 \[
 \begin{split}
\partial_p \bar\lambda_\mu&=  \partial_p \int\phi_\mu(y) \mu(\dd  y)= \int\partial_p \phi_\mu(y) \mu(\dd  y)\\
&= -a \int_0^{\overline t_f} e^{a(1-p)F_\mu(x)} F_\mu(x)f(x)
\Big ( \int_{(x,+\infty)} e^{-aF_\mu(y)} \mu(\dd y) \Big )\dd x \\
&\phantom{=}+
\int_0^{\overline t_f} f(x) \Big ( \int_{[0,x]} e^{-aF_\mu(y)} \mu(\dd y)\Big )  \dd x
\\
&= \int_0^{\overline t_f} f(x) \Big [ \int_{[0,x]} e^{-aF_\mu(y)} \mu(\dd y)
-a  e^{a(1-p)F_\mu(x)} F_\mu(x) \Big ( \int_{(x,+\infty)} e^{-aF_\mu(y)} \mu(\dd y) 
\Big ) \Big ]  \dd x \\
\end{split}
 \]
 Let us study the expression between brackets by means of the function $\bar h_p$ defined by the following equation
 \[
 \begin{split}
  h_p(x)&:= \int_{[0,x]} e^{-aF_\mu(y)} \mu(\dd y)
-a  e^{a(1-p)F_\mu(x)} F_\mu(x) \Big ( \int_{(x,+\infty)} e^{-aF_\mu(y)} \mu(\dd y) 
\Big )\\
&\ge \int_{[0,x]} e^{-aF_\mu(y)} \mu(\dd y)
-e^{-apF_\mu(x)} F_\mu(x) \big ( 1-e^{-a(1-F_\mu(x))} \big )=:\bar h_p(x)
\end{split}
\]
which holds for every $x \in [0,\overline t_f]$, $p \in [0,1]$ 
(the inequality comes from equation~\eqref{eq:domination1}).
 Clearly $p \mapsto \bar h_p(x)$ is nondecreasing hence $\bar h_p(x) \ge \bar h_0(x)$ for all $x \in [0, \overline t_f]$.
 In particular $p \mapsto \bar h_p(x)$ is strictly increasing if $F_\mu(x)>0$; in this case $\bar h_p(x) > \bar h_0(x)$.
 Suppose that $a \le 2 \log(2)$; thus $e^{-az}  + e^{-a(1-z)} \ge 2 e^{-a/2} \ge 1$.
 Now, since $F_\mu(0)=\mu(0)$,
 \[
  \bar h_0(0)=\mu(0) e^{-a\mu(0)}-\mu(0) \big (1-e^{-a(1-\mu(0))}\big )=\mu(0) \big (e^{-a\mu(0)}+e^{-a(1-\mu(0))}-1\big ) \ge 0.
 \]
 Moreover, if $ z \ge x \ge 0$ we have
 \[
 \begin{split}
 \bar h_0(z)-\bar h_0(x)&= \int_{(x,z]} e^{-aF_\mu(y)} \mu(\dd y)
-  F_\mu(z) \big (1- e^{-a(1-F_\mu(z))} \big ) + F_\mu(x) \big (1- e^{-a(1-F_\mu(x))} \big )\\
 &= \int_{(x,z]} e^{-aF_\mu(y)} \mu(\dd y)
 - (F_\mu(z)- F_\mu(x)) \big (1- e^{-a(1-F_\mu(z))} \big ) \\
 &\phantom{\mu(\dd y)}+ F_\mu(x) \big (e^{-a(1-F_\mu(z))} - e^{-a(1-F_\mu(x))} \big )\\
 & \ge (F_\mu(z)- F_\mu(x)) (e^{-a F_\mu(z)}+e^{-a(1-F_\mu(z))}-1) \\
 &\phantom{\mu(\dd y)}+ F_\mu(x) \big (e^{-a(1-F_\mu(z))} - e^{-a(1-F_\mu(x))} \big ) \ge 0
 \end{split}
 \]
 where the last inequality is strict if and only if $F_\mu(z)>F_\mu(x)$ 
 (we also used the fact that $\int_{(x,z]} e^{-aF_\mu(y)} \mu(\dd y) \ge e^{-aF_\mu(z)}\mu((x,z])=e^{-aF_\mu(z)}(F_\mu(z)- F_\mu(x))$).
 This means that, for all $x \in [0,\overline t_f]$, $p \in [0,1]$ we have $h_p(x) \ge \bar h_p(x) \ge \bar h_0(x) \ge \bar h_0(0) \ge 0$.
 Observe that $\partial_p \bar\lambda_\mu(a,p,\mu,f)=\int_0^{\overline t_f} f(x) h_p(x) \dd x$.
 Whence, for all $a \le 2 \log (2)$, for all $\mu$ and for all $f$, we have that 
$\partial_p \bar\lambda_\mu(a,p,\mu,f) \ge 0$ for all $p \in [0,1]$.
 
 Observe that $F_\mu(x)=0$ for all $x < \overline t_f$ if and only if $\overline t_f > \inf \mathrm{supp}(\mu)$.
 Hence, if $F_\mu(x)=0$ for all $x < \overline t_f$ clearly $\phi_\mu(y)=1$ for all $y \in \mathrm{supp}(\mu)$ which implies
 $\bar\lambda_\mu(a,p,\mu,f)=1$ for all $p \in [0,1]$.
 
 On the other hand if $F_{\mu}(x)>0$ for some $x < \overline t_f$ we have that $h_p(z)>0$ for every $z \in [x, \overline t_f]$, $p \in (0,1]$;
 thus
 for all $a \le 2 \log (2)$, for all $\mu$ and for all $f$, we have that 
$\partial_p \bar\lambda_\mu(a,p,\mu,f) > 0$ for all $p \in (0,1]$.

\end{proof}

\end{document}